\documentclass[a4paper, 10pt]{amsart}

\usepackage{a4wide}
\usepackage{amssymb, amsthm, amsmath}
\usepackage{graphicx}
\usepackage{bbm}
\usepackage{tikz}
\usepackage{subfigure}

\newcommand{\e}{\varepsilon}
\newcommand{\E}{\mathbb{E}}
\newcommand{\N}{\mathbb{N}}
\newcommand{\rr}{\mathbb{R}}
\newcommand{\supp}{\operatorname{\text{supp}}}
\renewcommand{\H}{\mathcal{H}}
\newcommand{\F}{\mathcal{F}}
\newcommand{\K}{\mathcal{K}}
\newcommand{\C}{\mathcal{C}}
\renewcommand{\P}{\mathcal{P}}
\newcommand{\Lip}{\text{Lip}}
\newcommand{\X}{\bar{X}}
\newcommand{\U}{\bar{U}}
\newcommand{\Z}{\bar{Z}}
\renewcommand{\S}{\mathcal{S}}
\newcommand{\I}{\mathcal{I}}
\newcommand{\R}{\mathcal{R}}
\newcommand{\M}{\mathcal{M}}
\newcommand{\T}{\mathcal{T}}
\DeclareMathOperator*{\argmax}{arg\,max}

\DeclareSymbolFont{bbold}{U}{bbold}{m}{n}
\DeclareSymbolFontAlphabet{\mathbbold}{bbold}
\newcommand{\1}{\mathbbold{1}}

\newtheorem{theorem}{Theorem}
\newtheorem{proposition}{Proposition}
\newtheorem{lemma}{Lemma}
\theoremstyle{definition}
\newtheorem{definition}{Definition}
\theoremstyle{remark}
\newtheorem{remark}{Remark}

\newtheorem{model}{{\bf Model}}

\title[Multiscale approach for disease dynamics]{A multiscale approach for spatially inhomogeneous disease dynamics}
\date{\today}

\author[Schmidtchen]{Markus Schmidtchen}
\address[Markus Schmidtchen]{\newline Department of Mathematics, Imperial College London, 
	\newline London SW7 2AZ, United Kingdom}
\email{m.schmidtchen15@imperial.ac.uk}

\author[Tse]{Oliver Tse}
\address[Oliver Tse]{\newline Department of Mathematics, Technische Universit\"at Kaiserslautern, 
	\newline Erwin-Schr\"odinger-Strasse, 67663 Kaiserslautern, Germany}
\email{tse@mathematik.uni-kl.de}

\author[Wackerle]{Stephan Wackerle}
\address[Stephan Wackerle]{\newline Department of Mathematics, Technische Universit\"at Kaiserslautern, 
	\newline Erwin-Schr\"odinger-Strasse, 67663 Kaiserslautern, Germany}
\email{steph.wackerle@gmail.com}

\begin{document}
	
\begin{abstract}
In this paper we introduce an agent-based epidemiological model that generalizes the classical SIR model by Kermack and McKendrick. We further provide a multiscale approach to the derivation of a macroscopic counterpart via the mean-field limit. The chain of equations acquired via the multiscale approach are investigated, analytically as well as numerically. The outcome of these results provide strong evidence of the models' robustness and justifies their applicability in describing disease dynamics, in particularly when mobility is involved.
\end{abstract}

\maketitle


\section{Introduction}\label{sec:intro}

The understanding of disease dynamics for the purpose of prevention and control has become extremely crucial in the recent years. The emergence and reemergence of infectious diseases such as influenza, HIV/AIDS, SARS, and more recently the sudden outburst of the Ebola and Zika virus are events of concern and interest to the general population throughout the world. Moreover, the environmental landscape in which we live is dynamic and often experiences dramatic shifts due to technological innovations that periodically alter the bounds of what we think is possible.

Mathematical models and computer simulations have become irreplaceable experimental tools for building and testing theories, assessing quantitative conjectures, answering specific questions, determining sensitivities to changes in parameter values, and estimating key parameters from data. Understanding the transmission characteristics of infectious diseases in communities, regions, and countries may lead to better approaches to decreasing the transmission of these diseases.

A classical epidemiology model is the renown SIR model formulated by Kermack and McKendrick in 1927 \cite{allen2008mathematical,brauer2001mathematical,hethcote2000mathematics,kermack1927contribution,murray2002mathematical}, which describes the spread of a disease among a single species of $N$ individuals. It is a compartmental model, i.e., the population is split up into three classes of individuals denoted by $S,I,R$, representing the total number of \emph{susceptible}, \emph{infected} and \emph{recovered} individuals, respectively. Since the effective time period is assumed to be sufficiently short, the model considers neither birth nor death phenomena, as well as migration of individuals. It further assumes that susceptible individuals $S$ have never been exposed to the disease, and that they may only be infected by contagious individuals. If $\beta>0$ (\emph{transmission rate}) denotes the average number of adequate contacts of a person per unit time, multiplied by the risk of infection, given contact between an infectious and a susceptible individual, and $\tau=1/\gamma$ is the mean waiting time until full recovery, then the SIR model reads
\[
 \frac{dS}{dt} = - {\beta} S I,\qquad \frac{dI}{dt} = {\beta} S I - \gamma I, \qquad \frac{dR}{dt} = \gamma I,
\]
supplemented with initial values $S(0) = S_0$, $I(0) = I_0$, $R(0) = R_0$ for some $S_0$, $I_0$, $R_0\in\rr_{\ge 0}$. Clearly the system conserves the number of individuals $N=S+I+R$ for all time $t\ge 0$. Since its introduction, there has been extensive work done on extending the model in various directions.

The understanding of human mobility plays a fundamental role to the research of vector-based and rapid geographical spread of emergent infectious diseases. A popular but rudimentary way to incorporate the spatial movement of hosts into epidemic models is to assume some type of host random movement, leading to reaction-diffusion type equations \cite{murray2003mathematical}. This strand of development was built on the pioneering work of Fisher in 1937, who used a logistic-based reaction-diffusion model to investigate the spread of an advantageous gene in a spatially extended population \cite{fisher1937wave}. For considering populations on large geographical scales, scientists have integrated topological features of traffic networks, such as highways, railways and air transportation into models for disease dynamics \cite{lewis2013dispersal}. Many of these models are stochastic in nature, which results from considering general random walks such as Brownian motion and L\'evy flights \cite{brockmann2009human,lewis2013dispersal}. 

Agent-based models and interacting particle systems have been widely used in understanding how order and stability, or a lack thereof, arises from the interaction of many agents \cite{liggett2012interacting}. In addition to the rigorous analysis of models that arise in statistical physics, biology, economics and, now, even in the physics of society, they also provide the means to predict global behavior of a system from the local dynamics between agents \cite{morale2001modeling,morale2005interacting}. Despite its simplicity, interacting particle systems may be easily extended to include highly complex interactions. Unfortunately, as the number of agents $N$ in the system increases, immense computational cost becomes inevitable. 

Multiscale modeling provides a way out. One begins by passing to the so-called {\em mean-field limit} $N\to\infty$, to obtain equations that describe the evolution of the probability density function $f_t$, over the possible states of the agents \cite{bolley2011stochastic,braun1977vlasov,carrillo2010particle,carrillo2010self,finkelshtein2010vlasov,finkelshtein2011vlasov,spohn2012large,sznitman1991topics}. From this probabilistic description, one may further characterize equations corresponding to average/statistical quantities, which have the tendency to be analytically well understood, and numerically tractable.

\begin{figure}[h]
    \centering
    \tikzstyle{block}=[draw,rectangle,fill=white, text width=2.5cm,minimum height=1cm,text centered,node distance=5em]
    \begin{tikzpicture} [auto]
        \node[block] (Mi) {{Microscopic \\ $(X_t,V_t)$}};
        \node[block] (Me)[right=10em] {{Mesoscopic \\ $f_t(x,v)$}};  
        \node[block] (Ma)[right=24em] {{Macroscopic}}; 
        \draw [->, thick] (Mi) --  node  {mean-field} (Me);	
        \draw [->, thick] (Me) --  node  {asymptotics} (Ma);
    \end{tikzpicture}
    \caption{Multiscale modeling.}\label{fig:hierarchy}
\end{figure}
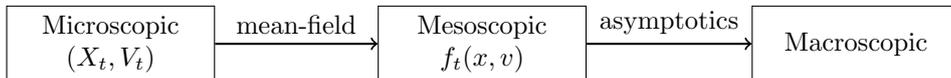
In this paper, we discuss a multiscale approach in deriving {\em macroscopic} epidemiology models from agent-based ({\em microscopic}) models. Fig.~\ref{fig:hierarchy} illustrates the general strategy in passing over from the microscopic to macroscopic regime, via the mesoscopic regime, thereby introducing the notion of multiscale \cite{klar2014approximate}. In Section~\ref{sec:micro} we introduce the agent-based (microscopic) model for epidemiology, that should generalize the classical SIR model in two ways, namely the inclusion of mobility and the continuous evolution of an agent's health status. The latter would allow for agents to resist an infection, which is absent from the classical SIR model. Section~\ref{sec:mean-field} provides an overview and some rigorous results pertaining to the limiting equation when the number of agents $N$ tends to infinity. We further provide a way to derive the corresponding macroscopic equation, which is a partial differential equation over the activity variable. In Section~\ref{sec:macro} we focus on the well-posedness of the macroscopic equation derived in the previous section and discuss possible stationary distributions. In this section we also point out a way to determine parameters, under which an epidemic may occur for the macroscopic model. Section~\ref{sec:numerics} is devoted to the numerical investigation of the models discussed in the previous sections. Here, we consider various scenarios that justify the adoption of the agent-based epidemiology model and its macroscopic counterpart to model disease dynamics in an spatially inhomogeneous environment. We finally conclude the paper in Section~\ref{sec:conclusion} with an outlook to future work and possible extensions of the models introduced within this paper.

\section{An Agent-Based Epidemiology Model}\label{sec:micro}

In this framework, we consider a system of $N\in\N$ identical agents with position $X_t^i\in\Omega\subset \rr^d$, where $\Omega$ is a domain, and {\em activity/health status} $U_t^i\in J\subset \rr$, $i=1,\ldots,N$ at time $t\in[0,\infty)$, satisfying for $i\in\{1,\ldots,N\}$, the system of stochastic differential equations 
\begin{align}\label{eq:micro}
 dX_t^i = \sqrt{2\sigma}\,dW_t^i,\qquad dU_t^i =-\H'(U_t^i)\,dt + F_N(X_t^i,U_t^i,{\bf X}_t,{\bf U}_t)\,dt,
\end{align}
with ${\bf X}_t=(X_t^1,\ldots,X_t^N)$, ${\bf U}_t=(U_t^1,\ldots,U_t^N)$, the {\em standard Wiener process} $W_t^i\in\rr^d$, and
\[
 F_N(X_t^i,U_t^i,{\bf X}_t,{\bf U}_t) = \frac{1}{N}\sum\nolimits_{j\ne i} \K(X_t^i, U_t^i, X_t^j, U_t^j),
\]
where $\H\colon J\to \rr$ is a given {\em potential landscape} describing the transition between two activity status, and $\K\colon S\times S\to \rr$ is the force describing {\em inter-agent interactions} on the state space $S:=\Omega \times J$. The initial configuration of the $N\in\N$ agents in $S^N$ are assumed to be independent and identically distributed random variables.

Mimicking the SIR model, we set $J=[-1,1]$. Then the activity $U_t^i$ is the internal variable describing an agent's activity, where $U_t^i=-1$ denotes the {\em susceptible state}, and $U_t^i=1$ the fully {\em recovered state}. This gives use the possibility to  work with a continuous health state, which mimics reality, since the infection of an individual intensifies or diminishes more or less continuously. As a first spatial extension to the SIR model and for clarity of presentation, the agents are only modelled to move randomly following the standard Wiener process with diffusion coefficient $\sqrt{2\sigma}$.

\begin{remark}\label{rem:weak_coupling}
 The factor $1/N$ in $F_N$ rescales the interaction force $\K$, and is typically called the {\em weak coupling scaling}, which will allow for the passage to mean-field. We will provide examples in Section~\ref{sec:numerics} for this case. Other forms of rescaling may be possible, but will not be discussed here.
\end{remark}

The underlying idea to prescribe $\H$ and $\K$ stems from {\em reaction rate theory} and is adapted to disease dynamics (cf.~\cite{hanggi1990reaction} and references therein). The following are three phenomenological features that are accounted for in our current model:
\begin{enumerate}
 \item A susceptible agent remains susceptible unless exposed to infectious agents. The exposure would need to exceed a certain threshold $u_* := \argmax_{u\in J} \H (u)$ for an agent to inherit the disease. This provides the possibility for an agent to resist the disease.
 \item If a susceptible agent is exposed to a sufficient amount of infectious agents over a duration of time, it will exceed the threshold $u_*$ and become infected. From that moment on, the recovery phase begins. The agent will be infectious for a period of time and then lose its ability to infect others when its activity exceeds some $\bar u\in (u_*,1]$. After some time it will arrive at the recovered state.
 \item Having reached the recovered state, the agent becomes immune to the disease. Therefore $u=1$ should be an attractor, i.e., local minimum of $\H$.
\end{enumerate}

For comparison with the classical SIR model, we devide the interval $J$ into compartments, indicating the current active health state. We denote the disjoint partition of $J$ by $\S=(-1,u_*)$, which represents the susceptible class, $\I=(u_*,\bar u)$, the infectious class, and $\R=(\bar u,1]$, the fully recovered class. The magnitude of each compartment is then measured simply by counting the number of agents located in the corresponding interval. 

Owing to these features mentioned above, we formalize them in the following definition.

\begin{definition}
 A potential landscape $\H$ is said to be feasible if $\H\in\Lip^1_b(J)$,
 \begin{enumerate}
  \item has two local minima at $u=1$ and $u=-1$, respectively, and
  \item has one global maximum at $u_*\in(-1,1)$.
 \end{enumerate}
 An inter-agent interaction force $\K$ is said to be feasible if $\K\in\Lip_b(S \times S)$, 
 \begin{enumerate}
  \item $\K(x,u,y,u)$ vanishes for any $x,y\in\Omega$, $u\in J$, and
  \item $\K(x,u,y,\nu)$ vanishes at $u\in\R\cup\{-1\}$ for any $x,y\in\Omega$, $\nu\in J$.
 \end{enumerate}
\end{definition}

\begin{remark}\label{rem:sus}
 Notice that we do not consider $u=-1$ as susceptible. This is due to the fact that, our feasible interaction force $\K$ does not allow interactions with agents that are at the state $u=-1$. Therefore, in this case, we can consider an agent at the state $u=-1$ to be immune to the disease.
\end{remark}

For completeness, we mention the solvability of our microscopic model (\ref{eq:micro}) for any finite number of agents $N\in\N$, which is an easy consequence of the strong existence and uniqueness for It\^o processes \cite{durrett1996stochastic}. In the following we set $Z_t^i=(X_t^i, U_t^i)$ for $1\le i\le N$ and denote $\P_p(S)$ to be the set of Borel probability measures with finite $p$-th moment.

\begin{proposition}\label{prop:micro}
 Let $T>0$ be arbitrary, $\H$ and $\K$ be feasible, and $N\in\N$. Furthermore, let the initial values $\{Z_0^i\}$ be mutually independent and $f_0$-distributed random variables with $f_0\in\P_2(S)$. Then there exists a unique, $t$-continuous, adapted solution ${\bf Z}_t=({\bf X}_t, {\bf U}_t)$ of the microscopic model (\ref{eq:micro}) with $\mathbb{E}\big[\int_0^T |{\bf Z}_t|^2\,dt\big] < \infty$.
\end{proposition}

\subsubsection*{Potential landscapes}
An exemplary class of potential landscapes $\H$ satisfying these features are known as {\em double-well potentials}, which includes, for example, potentials of the form
\[
 \H_{\alpha,\beta}(u) = \alpha(u^2-1)^2 + \beta(1-\sin(\pi u/2)),
\]
for suitable parameters $\alpha,\beta\ge0$. It is easy to see that the set of parameters $\alpha>0$, $\beta<2^5\alpha/\pi^2$, provides a feasible set of potential landscapes. One observes that the minima are local attractors. Therefore, an agent at the state $u$ in a neighborhood around $\{-1,1\}$ will remain there unless perturbed by a sufficient amount of external force. Moreover, an agent is said to have been \emph{infected} at some point of time if its activity has exceeded the threshold $\H(u_*)$. 

There are other possibilities for the potential landscape $\H$. For instance, one may use a smooth version of a piecewise affine linear function, as seen in Fig.~\ref{fig:potential}. Practically, the potential landscapes provided in Fig.~\ref{fig:potential} describes a disease that is easily contracted, and requires a relatively long time for complete recovery.

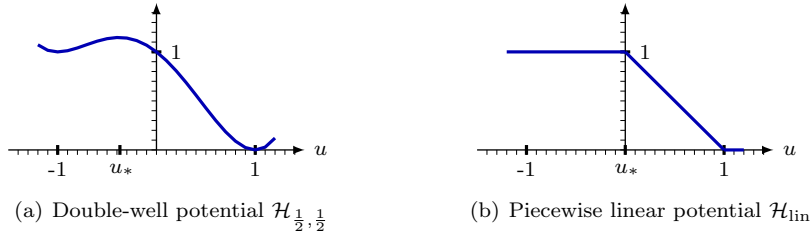
\begin{figure}[ht]
 \centering
 \subfigure[Double-well potential $\H_{\frac{1}{2},\frac{1}{2}}$]{
 \begin{tikzpicture}[line width=0.5pt, >=latex, xscale = 1.3, yscale=1.3]
  
  		\draw[->] (-1.5,0) -- (1.5,0) node[right] {\footnotesize$u$};
  		
  		\draw[->] (0,0) -- (0,1.5) ;
  
          \foreach \x in {-14,...,14}
          \draw [ultra thin] (0.1*\x,0) -- (0.1*\x,-0.05);
              
          \foreach \y in {0,...,14}
          \draw [ultra thin] (0,0.1*\y) -- (-0.05,0.1*\y);
          
  		  \draw[very thick] (-1,0.05) -- (-1,-0.05);
          \draw (-1.,-0.2) node {\footnotesize -1};
          
          \draw[very thick] (-0.37,0.05) -- (-0.37,-0.05);
          \draw (-0.35,-0.2) node {\footnotesize $u_*$};
          
          \draw[very thick] (1,0.05) -- (1,-0.05);
          \draw (1.,-0.2) node {\footnotesize 1};
          
          \draw[very thick] (-0.05,1.) -- (0.05,1.);
          \draw (0.2,1.) node {\footnotesize 1};
          
          \draw[blue!70!black, very thick, domain=-1.2:1.2] plot (\x, {(1/2)*(\x*\x-1)*(\x*\x -1) + (1/2)*(1-sin(deg(pi*\x*0.5)))} );         
  	\end{tikzpicture}
  }\hspace*{0.1\textwidth}
 \subfigure[Piecewise linear potential $\H_{\text{lin}}$]{
 \begin{tikzpicture}[line width=0.5pt, >=latex, xscale = 1.3, yscale=1.3]
  
  		\draw[->] (-1.5,0) -- (1.5,0) node[right] {\footnotesize$u$};
  		
  		\draw[->] (0,0) -- (0,1.5) ;
  
          \foreach \x in {-14,...,14}
          \draw [ultra thin] (0.1*\x,0) -- (0.1*\x,-0.05);
              
          \foreach \y in {0,...,14}
          \draw [ultra thin] (0,0.1*\y) -- (-0.05,0.1*\y);
          
  	   	  \draw[very thick] (-1,0.05) -- (-1,-0.05);
          \draw (-1.,-0.2) node {\footnotesize -1};
          
          \draw[very thick] (0,0.05) -- (0,-0.05);
          \draw (0.02,-0.2) node {\footnotesize $u_*$};
          
          \draw[very thick] (1,0.05) -- (1,-0.05);
          \draw (1.,-0.2) node {\footnotesize 1};
          
          \draw[very thick] (-0.05,1.) -- (0.05,1.);
          \draw (0.2,1.) node {\footnotesize 1};
          
          \draw[blue!70!black, very thick, domain=0:1] plot (\x, {(1-\x)} );
          \draw[blue!70!black, very thick, domain=-1.2:0] plot (\x, {1} );
          \draw[blue!70!black, very thick, domain=1:1.2] plot (\x, {0} );
  	\end{tikzpicture}
 }
 \caption{Examples of potential landscapes $\H$.}\label{fig:potential}
\end{figure}

\subsubsection*{Interaction forces}
The inter-agent interactions are described by the force $\K$, which typically depends on the distance $r=|X_t^i-X_t^j|$ between agents $i$ and $j$, and their corresponding activities $U_t^i$ and $U_t^j$. A product ansatz of the form
\[
 \K(x,u,y,\nu) = \Phi(x-y)\,\psi(u)\,\chi(\nu),\qquad (x,u),(y,\nu)\in S,
\]
may be used, since it easily captures the features mentioned above. Roughly speaking, the function $\Phi$ should be a non-negative even function, i.e., $\Phi(r)=\Phi(-r)$, which indicates if two agents are within close proximity for possible interactions to occur. Since the probability of infection is highest when particles are closest, we set $\Phi(0)=1$. On the other hand, the function $\psi$ indicates when an agent is susceptible to infection, whereas $\chi$ indicates when an agent is infectious. Based on our basic features, an infected agent should not be allowed to infect an agent that is already infected. Furthermore, an infected agent should not change the activity of a recovered agent. Therefore, a feasible $\K$ would have that $\supp(\psi)\cap\supp(\chi)=\emptyset$. Moreover, we require the supports of $\psi$ and $\chi$ to satisfy $\supp(\psi)\subset \S$ and $\supp(\chi)\subset \I$.

To conceive an easy and intuitive interaction force that fits the requirements for the model, we choose $\Phi(r)=\1_{B_R(0)}(r)$, $\psi = \1_\S$ and $\chi=c_\chi \1_\I$, where $c_\chi$ represents the strength of infection. Fig.~\ref{fig:interaction} provides an elementary example of an interaction force $\K$ made up of the functions $\Phi, \psi$ and $\chi$. For feasibility reasons, smooth versions of $\Phi$, $\psi$ and $\chi$ shown in the figure are used instead. 

\begin{figure}[ht]
    \centering
    \subfigure[Activation function $\Phi$]{
    \label{fig:interaction_potential_disc}
    \centering
	\begin{tikzpicture}[line width=0.5pt, >=latex, xscale = 1.35, yscale=1.35]
		\draw[->] (-1.5,0) -- (1.5,0) node[right] {\footnotesize$r$};
		
		\draw[->] (0,0) -- (0,1.5) ;

        \foreach \x in {-14,...,14}
            \draw [ultra thin] (0.1*\x,0) -- (0.1*\x,-0.05);
            
        \foreach \y in {0,...,14}
        \draw [ultra thin] (0,0.1*\y) -- (-0.05,0.1*\y);
        
		\draw[very thick] (-1,0.05) -- (-1,-0.05);
        \draw (-1.,-0.2) node {};
        
        \draw[very thick] (-0.5,0.05) -- (-0.5,-0.05);
        \draw (-0.5,-0.2) node {\footnotesize $-R$};
        
        \draw[very thick] (0.5,0.05) -- (0.5,-0.05);
        \draw (0.5,-0.2) node {\footnotesize $R$};        
        
        \draw[very thick] (1,0.05) -- (1,-0.05);
        \draw (1.,-0.2) node {};
        
        \draw[very thick] (-0.05,1.) -- (0.05,1.);
        \draw (0.2,1.) node {\footnotesize 1};
        
		\node at (-0.5,0.5) {\footnotesize\color{red!70!black}{$\Phi(r)$}};
        
        \draw[red!70!black, ultra thick, domain=-1:-0.5,smooth] plot (\x, {0});
        \draw[red!70!black, ultra thick, domain=-0.5:0.5,smooth] plot (\x, {1});
        \draw[red!70!black, ultra thick, domain=0.5:1,smooth] plot (\x, {0});
        
	\end{tikzpicture}
	}\quad
	\subfigure[Passive infection potential $\psi$]{
	    \label{fig:passive_infection_potential_disc}
	    \centering
		\begin{tikzpicture}[line width=0.5pt, >=latex, xscale = 1.35, yscale=1.35]
	
			\draw[->] (-1.5,0) -- (1.5,0) node[right] {\footnotesize$u$};
			
			\draw[->] (0,0) -- (0,1.5) ;
	
	        \foreach \x in {-14,...,14}
	        \draw [ultra thin] (0.1*\x,0) -- (0.1*\x,-0.05);
	            
	        \foreach \y in {0,...,14}
	        \draw [ultra thin] (0,0.1*\y) -- (-0.05,0.1*\y);
	        
			\draw[very thick] (-1,0.05) -- (-1,-0.05);
	        \draw (-1.,-0.2) node {\footnotesize -1};
	        
	        \draw[very thick] (-0.37,0.05) -- (-0.37,-0.05);
	        \draw (-0.40,-0.2) node {\footnotesize $u_*$};
	        
            \draw[very thick] (0.2,0.05) -- (0.2,-0.05);
            \draw (0.2,-0.2) node {\footnotesize $\bar u$};
	        
	        \draw[very thick] (1,0.05) -- (1,-0.05);
	        \draw (1.,-0.2) node {\footnotesize 1};
	        
	        \draw[very thick] (-0.05,1.) -- (0.05,1.);
	        \draw (0.2,1.) node {\footnotesize 1};
	        
			\node at (-0.5,0.5) {\footnotesize\color{red!70!black}{$\psi(u)$}};
	        
	        \draw[red!70!black, ultra thick, domain=-1:-0.37,smooth] plot (\x, {1});
	        \draw[red!70!black, ultra thick, domain=-0.37:0.2,smooth] plot (\x, {0});
	        \draw[red!70!black, ultra thick, domain=0.2:1,smooth] plot (\x, {0});
	        
		\end{tikzpicture}
	}\quad
    \subfigure[Active infection potential $\chi$]{
    \label{fig:active_infection_potential_disc}
    \centering
	\begin{tikzpicture}[line width=0.5pt, >=latex, xscale = 1.35, yscale=1.35]

		\draw[->] (-1.5,0) -- (1.5,0) node[right] {\footnotesize$\nu$};
		
		\draw[->] (0,0) -- (0,1.5) ;

        \foreach \x in {-14,...,14}
        \draw [ultra thin] (0.1*\x,0) -- (0.1*\x,-0.05);
            
        \foreach \y in {0,...,14}
        \draw [ultra thin] (0,0.1*\y) -- (-0.05,0.1*\y);
        
		\draw[very thick] (-1,0.05) -- (-1,-0.05);
        \draw (-1.,-0.2) node {\footnotesize -1};
        
        \draw[very thick] (-0.37,0.05) -- (-0.37,-0.05);
        \draw (-0.40,-0.2) node {\footnotesize $u_*$};
        
        \draw[very thick] (0.2,0.05) -- (0.2,-0.05);
        \draw (0.2,-0.2) node {\footnotesize $\bar u$};
        
        \draw[very thick] (1,0.05) -- (1,-0.05);
        \draw (1.,-0.2) node {\footnotesize 1};
        
        \draw[very thick] (-0.05,1.) -- (0.05,1.);
        \draw (0.2,0.98) node {\hspace*{1.5em}\footnotesize $c_\chi$};
        
        \node at (-0.5,0.5) {\footnotesize\color{red!70!black}{$\chi(u)$}};
        
        \draw[red!70!black, ultra thick, domain=-1:-0.37,smooth] plot (\x, {0});
        \draw[red!70!black, ultra thick, domain=-0.37:0.2,smooth] plot (\x, {1});
        \draw[red!70!black, ultra thick, domain=0.2:1,smooth] plot (\x, {0});
        
	\end{tikzpicture}
	}
	\caption{Possible choice of interaction force $\K(x,u,y,\nu) = \Phi(x-y)\,\psi(u)\,\chi(\nu)$.}\label{fig:interaction}
\end{figure}
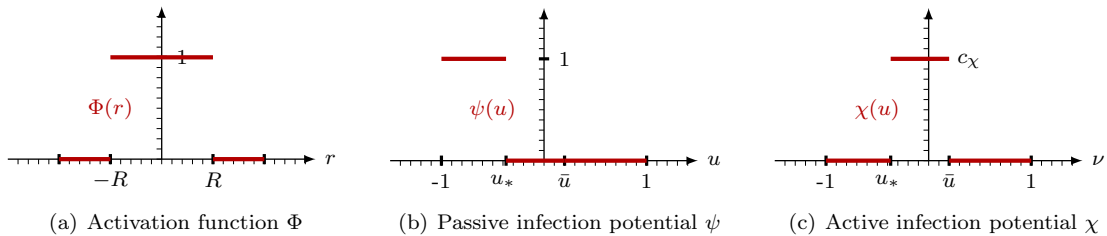

\section{The mean-field and macroscopic equations}\label{sec:mean-field}

In this section, we discuss the limiting process that appears when passing to the mean-field limit $N\to\infty$. The idea in obtaining a limiting equation is to replace the interaction term $F_N$, which depends on all binary interactions of any pair $(Z_t^i,Z_t^j)$, with an interaction term $\F$ that describes the interaction of a single agent with an averaged field, the so-called mean-field $f_t$. Intuitively, if one considers the empirical measure $\mu_t^N$ of the stochastic processes $\{Z_t^i\}$ given by
\[
 \mu_t^N(dz) = \frac{1}{N}\sum\nolimits_{i=1}^N \delta_{Z_t^i}(dz),
\]
where $\delta_{z}$ denotes the Dirac measure at $z\in S$, then one could formally write
\[
 F_N(Z_t^i,{\bf Z}_t) = \frac{1}{N}\sum\nolimits_{j\ne i} \K(Z_t^i, Z_t^j) = \int_{S} \K(Z_t^i,z') \mu_t^N(dz').
\]
One then strives to show that the empirical measure $\mu_t^N$ converges to a deterministic limit measure $f_t$ in the sense of convergence in law for the underlying random variables. In this case, one may then show that
\[
 F_N(Z_t^i,{\bf Z}_t) \;\longrightarrow\;\F[f_t](Z_t^i) = \int_{S} \K(Z_t^i,z')f_t(dz'),
\]
in some appropriate notion of convergence, as $N\to\infty$.

\subsection{Nonlinear process and mean-field equations}
Indeed, the conjecture is that the process $Z_t^i$ for some fixed $1\le i\le N$ generated by the microscopic system (\ref{eq:micro}) converges to a mean-field process, the so-called {\em McKean nonlinear process}, given by the solution of 
\begin{align}\label{eq:nonlinear}
 d\X_t^i = \sqrt{2\sigma}\,dW_t^i,\qquad d\U_t^i = -\H'(\U_t^i)\,dt + \F[f_t](\X_t^i,\U_t^i)\,dt,
\end{align}
with the same family of standard Wiener processes $\{W_t^i\}$ as in (\ref{eq:micro}), and
\[
 \F[f_t](z) = \int_{S} \K(z,z')f_t(dz'),
\]
where $f_t=\text{law}(\Z_t^i)$ is the law of the random variable $\Z_t^i=(\X_t^i,\U_t^i)$. Since the law $f_t$ is required in the definition of the process $\Z_t^i$, we have a nonlinear stochastic system at hand. This nonlinear process is supplemented with mutually independent $f_0$-distributed initial conditions $\Z_0^i$. Note that the solutions $\{\Z_t^i\}$ are also independent and identically distributed with the joint law $f_t^{\otimes N}$.

Applying It\^{o}'s formula to the nonlinear process provides an evolution equation for their common law $f_t$, given by
\begin{align}\label{eq:fp}
 \partial_t f_t - \partial_u (\H'f_t - \F[f_t]f_t)  = \sigma\Delta_x f_t,\qquad \lim\nolimits_{t\searrow 0} f_t = f_0.
\end{align}
This nonlinear and nonlocal kinetic equation is commonly known as the {\em Fokker--Planck} equation corresponding to the nonlinear process (\ref{eq:nonlinear}).

As in the microscopic case, we recall an existence and uniqueness result for the nonlinear process (\ref{eq:nonlinear}), as well as the nonlinear kinetic equation (\ref{eq:fp}). Under the feasibility assumptions on $\H$ and $\K$, the proof of the following result is rather standard and may be found, for example, in \cite{bolley2011stochastic,sznitman1991topics}.

\begin{proposition}\label{prop:nonlinear}
 Let $T>0$ be arbitrary, $\H$ and $\K$ be feasible, $f_0\in \P_2(S)$. Then the nonlinear process (\ref{eq:nonlinear}) has a pathwise unique solution $\Z\in\C([0,T),S)$, with $f_t=\text{law}(\Z_t)\in \P_2(S)$, $t\in[0,T)$, satisfying the Fokker--Planck equation (\ref{eq:fp}).
\end{proposition}

Having unique strong solutions corresponding to (\ref{eq:micro}) and (\ref{eq:nonlinear}), we may provide a quantitative estimate of the difference between the two solutions $Z_t^i$ and $\Z_t^i$ for any $1\le i\le \N$, and consequently also the difference between their respective laws. For completeness, we provide the proof of the following theorem in Appendix~\ref{append:proof}.

\begin{theorem}\label{thm:mean-field}
 Let $T>0$ be arbitrary, $\H$ and $\K$ be feasible, and $f_0\in\P_2(S)$. Consider the solutions $Z_t^i$, $\Z_t^i$ to the equations (\ref{eq:micro}), (\ref{eq:nonlinear}) for $t\in[0,T]$ and each $1\le i\le N\in\N$ with mutually independent $f_0$-distributed $Z_0^i$, $\Z_0^i$, provided by Propositions~\ref{prop:micro} and \ref{prop:nonlinear} respectively. Then there exists a constant $C>0$, independent of $N\in\N$, such that
\begin{align}\label{eq:chaos}
 \sup\nolimits_{t\in[0,T]}\E[|Z_t^i - \Z_t^i|^2] \le CN^{-1},
\end{align}
for any $1\le i\le N\in\N$.
\end{theorem}

 The property of the stochastic empirical measure becoming deterministic in the limit is equivalent to the requirement that the law of the $N$ particles become {\em chaotic} in the limit \cite{sznitman1991topics}. This means that, for a fixed $k$, the law of the first $k$ agents $f_t ^{(k)}$ satisfies
 \[
  f_t ^{(k)} \longrightarrow f_t^{\otimes k}\quad\text{in}\;\;\P(S^k),
 \]
 as $N\to \infty$, assuming the $k$ agents to be initially $f_0^{\otimes k}$-distributed.
 
 In fact, estimate (\ref{eq:chaos}) ensures both theses properties:
\begin{enumerate}
 \item {\em Propagation of chaos property}. Indeed, we deduce from \eqref{eq:chaos} the estimate
 \begin{align*}
  W_2^2(f_t^{(k)},f_t^{\otimes k}) \le \E[|(Z_t^1,\ldots,Z_t^k)-(\Z_t^1,\ldots,\Z_t^k)|^2] \le k\,CN^{-1},
 \end{align*}
 where $W_2$ denotes the Wasserstein distance between measures in $\P_2(S)$ defined by
 \[
  W_2(\mu,\bar\mu) = \inf\nolimits \sqrt{\E[|Z-\Z|^2]}.
 \]
 The infimum is taken over all coupling of random variables $(Z,\Z)$ in $S\times S$ having distributions $\mu$ and $\bar\mu$ respectively (cf.~\cite{villani2008optimal}).
 
 \item {\em Convergence of the stochastic empirical measure $\mu_t^N$ towards the deterministic mean-field distribution $f_t$}. Due to \eqref{eq:chaos}, we have for any $\varphi\in\Lip_b(S)$ the estimate
 \begin{align*}
  \E\left[ \frac{1}{N}\sum\nolimits_{i=1}^N \varphi(Z_t^i) - \int_S \varphi f_t(dz)\right] &\\
  &\hspace*{-12em}\le 2\E\left[ \frac{1}{N}\sum\nolimits_{i=1}^N|\varphi(Z_t^i) - \varphi(\Z_t^i)|^2 + \left| \frac{1}{N}\sum\nolimits_{i=1}^N \varphi(\Z_t^i) - \int_S \varphi f_t(dz) \right|^2 \right] \le CN^{-1},
 \end{align*}
 for some constant $C>0$ independent of $N$ and $t\in[0,T]$. Notice that the second term in the first inequality follows from the law of large numbers. Indeed, this holds since $\{\Z_t^i\}$ are mutually independent and identically distributed (see also Appendix~\ref{append:proof}).
\end{enumerate}

Rigorous results of this form were known for the deterministic case since the 70's \cite{braun1977vlasov,dobrushin1979vlasov,spohn2012large}, and then extended to the stochastic case in \cite{bolley2011stochastic,sznitman1991topics}, see also \cite{mckean1967propagation}.

\subsection{Macroscopic equations}
At this point, one may derive equations governing macroscopic quantities based on the moments of $f_t$ by introducing closure relations or further assumptions on $\H$ and $\K$. In the following, we assume $f_t$ to have a sufficiently smooth density with respect to the Lebesgue measure on $S$, which we denote again by $f_t$. 

The following are several examples that may be of interest:

\begin{model}\label{ex:heat}
The zeroth order moment of $f_t$ w.r.t.~$u$, i.e., the first marginal of $f_t$:
 \[
  \rho_t = \int_{J} f_t(\cdot,u)\,du,
 \]
 satisfies the simple {\em heat equation}
 \[
  \partial_t \rho_t = \sigma\Delta_x\rho_t,\qquad \lim\nolimits_{t\searrow 0} \rho_t = \rho_0,
 \]
 which precisely describes the purely diffusive behavior of the nonlinear stochastic process in its first component, namely $\X_t$. Indeed, for feasible $\H$ and $\K$, we have
 \[
  \int_{J} \partial_u (\H'f_t - \F[f_t]f_t)\,du = (\H'f_t - \F[f_t]f_t )\Big|_{-1}^{1} = 0.
 \]
 If we consider the Wiener process $W_t$ in a bounded domain $\Omega\subset\rr^d$ with reflective boundary conditions, i.e., we allow the motion of agents only within a bounded region $\Omega$, we obtain the homogeneous Neumann boundary condition for the heat equation. In this case, the unique equilibrium for this equation is the constant $\rho_{\text{stat}}\equiv 1/|\Omega|$, i.e., the uniform distribution in the $x$-variable.
\end{model}
 
 \begin{remark}\label{rem:ornstein}
  Instead of considering a bounded domain $\Omega\subset \rr^d$ with reflecting boundary conditions for the Wiener process, one could introduce a sufficiently smooth and convex confining potential $V\colon\rr^d\to \rr$ with a sufficiently strong growth condition, and additionally $\int_{\rr^d} e^{-V/\sigma}\,dx = 1$. In this case, the mean-field spatial process becomes
  \[
   d\X_t = -\nabla_x V(\X_t)\,dt + \sqrt{2\sigma}\,dW_t,
  \]
  and the resulting Fokker--Planck equation reads
  \[
   \partial_t f_t - \partial_u (\H'f_t - \F[f_t]f_t)  = \text{div}_x(\sigma \nabla_x f_t + f_t\nabla_x V ).
  \]
  As in Model~\ref{ex:heat}, we may take the first marginal of $f_t$ to obtain
  \[
   \partial_t \rho_t = \text{div}_x(\sigma \nabla_x \rho_t + \rho_t\nabla_x V ),
  \]
  which is the Fokker--Planck equation corresponding to the {\em Ornstein--Uhlenbeck process}. Its unique stationary state is simply given by $\rho_{\text{stat}}=e^{-V/\sigma}$.
 \end{remark}
 
\begin{model}
 Disintegrating the joint probability distribution $f_t=f_t(x,u)$ into its first marginal $\rho_t$ and the corresponding conditional distribution $g_t^x$, i.e., $f_t(x,u)=g_t^x(u)\rho_t(x)$, and inserting this into the mean-field equation yields
 \begin{align}\label{eq:disintegrate}
  \partial_t g_t^x - \partial_u \Big(\H'g_t^x - \F[f_t]g_t^x\Big) = \sigma\Big( \Delta_x g_t^x + \nabla_x \ln\rho_t^2\cdot\nabla_x g_t^x\Big),
 \end{align}
 which is a closed equation for $g_t^{x}$, given $\rho_t$. In fact, if one is given a stationary spatial density $\rho_{\text{stat}}$ of the population, this can be included directly by simply setting $\rho_t=\rho_{\text{stat}}$. Notice that this equation is nonlocal in the spatial variable, unless further assumptions are made.
 
 Nevertheless, \eqref{eq:disintegrate} allows for the computation of $g_t^x$ for any given spatial distribution $\rho_t$, i.e., also those that do not necessarily satisfy the heat equation. Therefore, this macroscopic equation is capable of describing disease dynamics in spatially inhomogeneous populations, where the spatial inhomogeneity is provided by an arbitrary time dependent spatial distribution $\rho_t$.
\end{model}

\begin{model}\label{model:macro}
 A crude approximation to localize the spatial variable in \eqref{eq:disintegrate} would be to neglect the spatial derivatives on the right-hand side, and to use the product ansatz for the interaction term of the form $\K(x,u,y,\nu)=\delta_x(y)\Psi(u)\chi(\nu)$, which may be justified in the following sense. Suppose we rescale the spatial variable as $\tilde x\sim \e x$ and the density as $\tilde f_t \sim f_t(\cdot/\e)$, i.e., we assume that the spatial domain $\Omega$ is large in comparison to the range of interaction given by $\Phi$. Then, we obtain the scaled equation (dropping the tildes)
 \[
  \partial_{t} g_{t}^{x} -\partial_u \Big(\H' g_{t}^{x} -  \F_\e[f_t]g_{t}^{x}\Big) = \e^2\sigma\Big( \Delta_{x} g_{t}^{x} + \nabla_{x} \ln \rho_{t}^2\cdot\nabla_{x} g_{t}^{x}\Big),
 \]
 where
 \[
  \F_\e[f_t](x,u) = \int_{S} \e^{-d}\,\Phi(y/\e)\psi(u)\chi(\nu) f_t(x-y,\nu)\,dyd\nu.
 \]
 Now, if $\Phi$ has a form of a mollifier, then $\e^{-d}\,\Phi(\cdot/\e)$ converges towards the Dirac $\delta_0$ in distribution. By assuming $f_t$ to be sufficiently smooth, we may formally pass to the limit $\e\to 0$ to obtain
 \begin{align}\label{eq:conditional}
  \partial_t g_t^x -\partial_u \Big(\H' g_t^x - \rho_t \,\M[g_t^x]g_t^x\Big) = 0,
 \end{align}
 with the nonlocal (in the activity variable $u$) operator
 \begin{align}\label{eq:nonlocal_M}
  \M[g](u) = \psi(u)\int_{J}\chi(\nu)g(\nu)\,d\nu.
 \end{align}
 This localization procedure in the spatial variable provides a pointwise description of the activity, which, from the numerical point of view, is advantageous over the complete mean-field equation, since solving for $g_t^x$ with respect to the spatial variable $x\in\Omega$ may be carried out in a pointwise manner, independent of the activity variable $u\in J$.
\end{model}
 
\begin{model}
Assuming further that $\rho_t\equiv\rho_{\text{stat}}=1/|\Omega|$, which is the case in spatially homogeneous epidemiology models, a spatially independent model is recovered in the form
\begin{align}\label{eq:nonlocal_g}
 \partial_t g_t -\partial_u \Big(\H'g_t - \rho_{\text{stat}}\M[g_t]g_t\Big) = 0,
\end{align}
 which describes the probability distribution only in the activity variable $u\in J$. One can further extract other relevant information, such as the probability of finding agents with a certain activity set $A\in\mathcal{B}(J)$, simply given by $\int_A g_t \,du\in[0,1]$. For example, by choosing $A=\S$, we recover the probability of finding particles that are susceptible at time $t\ge 0$.
\end{model}

The last two equations \eqref{eq:conditional}, \eqref{eq:nonlocal_g} are the simplest of the macroscopic equations. Nevertheless, they sufficiently exhibit important characteristics of a basic epidemiological model modulo the spatial resolution. For this reason, we will study these equations in more detail in the next section.

\section{The nonlocal spatially homogeneous macroscopic equation}\label{sec:macro}

In this section, we provide an analytical study of the macroscopic equation
\[
 \partial_t g_t^x -\partial_u \Big(\H' g_t^x - \bar \rho(x) \,\M[g_t^x]g_t^x\Big) = 0,\qquad \lim\nolimits_{t\searrow 0} g_t^x = g_0^x,\quad x\in\Omega,
\]
where $\bar\rho$ is a given stationary smooth spatial distribution, and $\M$ is as given in \eqref{eq:nonlocal_M}. Since this equation may be solved pointwise in $x\in\Omega$, we consider the simpler variant
\begin{align}\label{eq:macro}
 \partial_t g_t -\partial_u \Big(\H' g_t - \bar \rho \,\M[g_t]g_t\Big) = 0,\qquad \lim\nolimits_{t\searrow 0} g_t = g_0,
\end{align}
where $\bar\rho$ is simply a constant, $\bar\rho >0$, and $g_0$ is an initial distribution of activity.

\subsection{Existence and uniqueness}
The well-posedness of a nonlocal continuity equation such as \eqref{eq:macro} may be found, for example, in \cite{crippa2013existence}. Nevertheless, for the convenience of the reader, we provide the principal ideas behind the solvability of the equation.

As in the standard {\em method of characteristics} for first order partial differential equations, we may derive the {\em characteristic equation} corresponding to the continuity equation \eqref{eq:macro}, which reads
\begin{align}\label{eq:characteristic}
 \frac{d}{dt}U_t(u) = -\H'(U_t(u)) + \bar\rho\M[g_t](U_t(u)),\qquad U_0(u)=u\in J.
\end{align}
One recognizes that this equation is again of the form of a nonlinear process since the flow $U_t$ depends on its density $g_t$. In fact, if $U_t$ satisfies the characteristic equation \eqref{eq:characteristic}, then its density $g_t$ may be represented by the push-forward of the flow $U_t$, i.e., $g_t = U_t\# g_0$, or equivalently
\[
 \int_J \varphi(u) dg_t = \int_J (\varphi\circ U_t)(v)\,dg_0\qquad\text{for all\, $\varphi\in \C_b(J)$}.
\]

Therefore, we define the notion of a {\em Lagrangian solution} of \eqref{eq:macro} with initial data $g_0\in\P_1(J)$  as a probability measure $g\in\C([0,T],\P_1(J))$ satisfying the push-forward formula $g_t=U_t\#g_0$ with the flow $U\in\C([0,T]\times J,J)$ satisfying \eqref{eq:characteristic}. It is known that Lagrangian solutions and weak measure solutions for \eqref{eq:macro} coincide (cf.~\cite{crippa2013existence}). The main result of this section is the following theorem.

\begin{theorem}\label{thm:macro}
 Let $T>0$ be arbitrary, $g_0\in\P_1(J)$, and $\H$ be feasible and $\psi$, $\chi\in\Lip_b(S)$. Then, there exists a unique Lagrangian solution $g\in\C([0,T],\P_1(J))$ to the equation \eqref{eq:macro}.
\end{theorem}

Its proof relies on the use of the well-known {\em Banach fixed point theorem} \cite{zeidler1993nonlinear} for complete metric spaces. For this reason, we consider the space $\C([0,T],\P_1(J))$, endowed with the distance
\[
 d(\mu,\nu) = \sup\nolimits_{t\in[0,T]} W_1(\mu_t,\nu_t),\qquad \mu,\nu\in \C([0,T],\P_1(J)),
\]
where $W_1$ denotes the 1-Wasserstein distance, given by
\[
 W_1(\mu,\nu) = \inf\nolimits_{\pi\in\Pi(\mu,\nu)}\iint_{J\times J} |x-y|\,d\pi(x,y),\qquad \mu,\nu\in\P_1(J).
\]
Here $\Pi(\mu,\nu)$ denotes the set of all measures $\pi$ with marginals $\pi(\cdot,J)=\mu$ and $\pi(J,\cdot)=\nu$. It is known that the 1-Wasserstein distance metricizes the narrow convergence in $\P_1(J)$, which makes $(\P_1(J),W_1)$ a separable complete metric space, since $J$ is complete \cite{villani2008optimal}. Consequently, the function space $\C([0,T],\P_1(J))$ endowed with the distance $d$ above is also a separable complete metric space.

Now consider, for any given $\hat g\in \C([0,T],\P_1(J))$, the auxiliary problem 
\begin{align}\label{eq:aux}
 \frac{d}{dt}U_t(u) = -\H'(U_t(u)) + \bar\rho\M[\hat g_t](U_t(u)),\qquad U_0(u)=u\in J.
\end{align}
It is easy to see that the right-hand side of the equation is continuous in the temporal variable, and globally Lipschitz-continuous in the activity variable for any feasible functions $\H$, $\psi$ and $\chi$. Therefore, the {\em Picard--Lindel\"of theorem}, or similarly, the {\em Cauchy--Lipschitz theorem}, provides a unique global solution $U_{\cdot}(u)\in \C([0,T],J)$ for any $u\in J$, and thereby a flow $U\in \C([0,T]\times J,J)$. We then construct a new probability measure $g\in \C([0,T],\P_1(J))$ by means of push-forward, i.e., $g_t = U_t\# g_0$, where $g_0\equiv \hat g_0\in \P_1(J)$.

Consequently, this induces a mapping $\T\colon\C([0,T],\P_1(J))\to \C([0,T],\P_1(J))$, $\hat g\mapsto g$, which we show to admit a fixed point satisfying the nonlocal continuity equation \eqref{eq:macro}. Before proceeding with the proof of Theorem~\ref{thm:macro}, we provide a stability estimate that will assist in showing the required contracting property of the mapping $\T$.

\begin{lemma}\label{lem:stability}
 Let $\hat g,\hat h\in \C([0,T],\P_1(J))$ be given and $g$, $h$ be Lagrangian solutions to
 \[
  \partial_t g_t -\partial_u \Big(\H' g_t - \bar \rho \,\M[\hat g_t]g_t\Big) = 0,\qquad \partial_t h_t -\partial_u \Big(\H' h_t - \bar \rho \,\M[\hat h_t]h_t\Big) = 0,
 \]
 with initial conditions $g_0=\hat g_0$ and $h_0=\hat h_0$ in $\P_1(J)$, respectively. Then the estimate
 \[
  W_1(g_t,h_t) \le \left(W_1(g_0,h_0) + c_2 \int_0^t W_1(\hat g_s,\hat h_s)\,ds\right)e^{c_1 t}\qquad \text{for all\, $t\ge 0$},
 \]
 holds true with positive constants $c_1,c_2$, depending only on $\bar\rho$, $\H$, $\psi$ and $\chi$.
\end{lemma}
\begin{proof}
 We first note that $g_t=U_t\#g_0$ and $h_t=V_t\# h_0$, where $U,V\in\C([0,T]\times J,J)$ satisfy
 \begin{align*}
  U_t(u) &= u -\int_0^t \H'(U_s(u)) - \bar\rho\M[\hat g_s](U_s(u))\,ds,\\
  V_t(u) &= v -\int_0^t \H'(V_s(v)) - \bar\rho\M[\hat h_s](V_s(v))\,ds,
 \end{align*}
 respectively. Now let $\pi_0\in \Pi(g_0,h_0)$ be an optimal coupling of $g_0$ and $h_0$, and $\pi_t = (U_t,V_t)\#\pi_0$. Then $\pi_t\in \Pi(U_t\#g_0,V_t\#h_0)=\Pi(g_t,h_t)$, which is not necessarily optimal. For $\pi_t$, we have
 \begin{align*}
  W_1(g_t,h_t) &\le \iint_{J\times J} |u-v|\,d\pi_t(u,v) = \iint_{J\times J} |U_t(u)-V_t(v)|\,d\pi_0(\hat u,\hat v) \\
  &\le \iint_{J\times J} |u-v|\,d\pi_0(u,v) + \int_0^t\iint_{J\times J}  |\H'(U_s(u))-\H'(V_s(v))|\,d\pi_0(u,v)\,ds \\
  &\hspace*{6em}+ \bar\rho\int_0^t\iint_{J\times J}  |\M[\hat g_s](U_s(u))-\M[\hat h_s](V_s(v))|\,d\pi_0(u,v)\,ds \\
  &= W_1(h_0,g_0) + I_1 + I_2.
 \end{align*}
 To estimate $I_1$, we simply use the Lipschitz-continuity of $\H'$ to obtain
 \begin{align*}
  I_1 \le L_{\H'}\int_0^t \iint_{J\times J} |U_s(u)-V_s(v)|\,d\pi_0(u,v)\,ds = L_{\H'}\int_0^t \iint_{J\times J} |u-v|\,d\pi_s(u,v)\,ds.
 \end{align*}
 Similarly, we use the Lipschitz-continuity of $\psi$ to obtain
 \begin{align*}
  I_2 &= \bar\rho\int_0^t\iint_{J\times J}  \left|\psi(U_s(u))\left(\int_J \chi(\hat u)\, d\hat g_s(\hat u)\right) -\psi(V_s(v))\left(\int_J \chi(\hat v)\, d\hat h_s(\hat v)\right)\right|\,d\pi_0(u,v)\,ds \\
  &\le \bar\rho L_\psi\|\chi\|_\infty\int_0^t\iint_{J\times J}  |U_s(u)-V_s(v)|\,d\pi_0(u,v)\,ds\\
  &\hspace*{12em}+\bar\rho\|\psi\|_\infty \int_0^t \left|\int_J \chi(\hat u)\, d\hat g_s(\hat u)-\int_J \chi(\hat v)\, d\hat h_s(\hat v)\right|\,ds,
 \end{align*}
 where we used the fact that $\psi$ and $\chi$ are bounded, and $\hat g_t$ and $\pi$ are a probability measures over $J$ and $J\times J$, respectively. Concerning the last term, we estimate further to obtain
 \begin{align*}
  \left|\int_J \chi(\hat u)\, d\hat g_s(\hat u)-\int_J \chi(\hat v)\, d\hat h_s(\hat v)\right| &\le \iint_{J\times J} |\chi(\hat u)-\chi(\hat v)|\, d(\hat g_s \otimes \hat h_s)(\hat u,\hat v) \\
  &\le L_\chi \iint_{J\times J} |\hat u-\hat v| d(\hat g_s \otimes \hat h_s)(\hat u,\hat v).
 \end{align*}
 Putting all the terms together yields
 \begin{align*}
  W_1(g_t,h_t) &\le W_1(h_0,g_0) + \big(L_{\H'}+\bar\rho L_\psi\|\chi\|_\infty\big)\int_0^t \iint_{J\times J} |u-v|\,d\pi_s(u,v)\,ds \\
  &\hspace*{12em}+ \bar\rho L_\chi\|\psi\|_\infty \int_0^t \iint_{J\times J} |u-v|d(\hat g_s\otimes \hat h_s)(u,v)\,ds.
 \end{align*}
 Optimizing the right-hand side over all possible couplings in $\Pi(g_s,h_s)$ and $\Pi(\hat g_s,\hat g_s)$ gives
 \[
  W_1(g_t,h_t) \le W_1(g_0,h_0) + c_1\int_0^t W_1(g_s,h_s)\,ds + c_2 \int_0^t W_1(\hat g_s,\hat h_s)\,ds,
 \]
 with $c_1 = L_{\H'}+\bar\rho L_\psi\|\chi\|_\infty$ and $c_2 = \bar\rho L_\chi\|\psi\|_\infty$. From {\em Gronwall's inequality}, we finally obtain
 \[
  W_1(g_t,h_t) \le \left(W_1(g_0,h_0) + c_2 \int_0^t W_1(\hat g_s,\hat h_s)\,ds\right)e^{c_1 t},
 \]
 which completes the proof.
\end{proof}

We now have all the ingredients necessary to complete the proof of Theorem~\ref{thm:macro}.

\begin{proof}[Proof of Theorem~\ref{thm:macro}]
 We consider the mapping $\T\colon \C([0,T],\P_1(J))\to \C([0,T],\P_1(J))$ as discussed above. However, we consider a weighted metric of the form
 \[
  d_\lambda(g,h) = \sup\nolimits_{t\in[0,T]} e^{-\lambda t}W_1(g_t,h_t),
 \]
 which is clearly equivalent to the usual metric $d$, for any $\lambda>0$. Therefore, the space $\C([0,T],\P_1(J))$ endowed with the metric $d_\lambda$ is again a separable complete metric space.
 
 Now let $g=\T(\hat g)$ and $h=\T(\hat h)$, with $g_0=h_0$. Then Lemma~\ref{lem:stability} provides the estimate
 \begin{align*}
  W_1(g_t,h_t) &= c_2 e^{c_1 t} \int_0^t W_1(\hat g_s,\hat h_s)\,ds \le (c_2e^{c_1 t}/\lambda) (e^{\lambda t}-1)d_\lambda(\hat g,\hat h).
 \end{align*}
 Multiplying both sides by $\exp(-\lambda t)$ and taking the supremum over time $t\in[0,T]$ yields
 \[
  d_\lambda(\T(\hat g),\T(\hat h))\le (c_2e^{c_1 T}/\lambda) d_\lambda(\hat g,\hat h).
 \]
 Therefore, choosing $\lambda > c_2\exp(c_1 T)$ makes $\T$ a contraction mapping with respect to the metric $d_\lambda$. Finally, we invoke the Banach fixed point theorem to obtain a unique fixed point in the space $\C([0,T],\P_1(J))$, which satisfies the nonlocal macroscopic equation \eqref{eq:macro}.
\end{proof}

In fact, one can further show that if the initial measure $g_0\in\P_1^{\text{ent}}(J)$, where $\P_1^{\text{ent}}(J)$ denotes the space of probability measures with finite first moment that are, additionally, absolutely continuous with respect to the Lebesgue measure and have finite entropy
\[
 0\le \text{Ent}(g)=\int_J \big(g\log(g)-g+ 1 \big)du < \infty,\qquad g\in \P_1^{\text{ent}}(J),
\]
then $g_t\in\P_1^{\text{ent}}(J)$ for all times $t\ge 0$. Indeed, assuming $\H$, $\psi$ and $\chi$ to be feasible, then
\begin{align*}
 \frac{d}{dt}\text{Ent}(g_t)
 &= -\int_J \partial_u g_t \Big(\H'(u) - \bar\rho\M[g_t]\Big)du = \int_J g_t\Big(\H''(u) - \bar\rho\partial_u\M[g_t]\Big)du \\
 &\le \big(1+\bar\rho\|\chi\|_\infty\big)\text{Ent}(g_t) + \int_J \Big(e^{\H''(u)}-1\Big) du + \bar\rho\|\chi\|_\infty\int_J \Big(e^{\psi'(u)}-1\Big)du \\
 &= c_1\text{Ent}(g_t) + c_2,
\end{align*}
where we integrated by parts in the first two equalities, and applied Young's inequality of the form $ab \le e^{a} + b\ln(b)-b$ for $a,b\in\rr$, $b\ge 0$ in the inequality. Equivalently, we have in integral form
\[
 \text{Ent}(g_t) \le (\text{Ent}(g_0) + c_2 t) + c_1\int_0^t \text{Ent}(g_s)\,ds.
\]
A simple application of the Gronwall inequality leads to the estimate
\[
 \text{Ent}(g_t) \le (\text{Ent}(g_0) + c_2 t)e^{c_1 t},
\]
which shows that $g_t\in \P_1^{\text{ent}}(J)$ for all times $t\ge 0$ as asserted. Summarizing, we have

\begin{proposition}
 Let $\H$ be feasible, $\psi$, $\chi\in\Lip_b(S)$ and $g_0\in \P_1^{\text{ent}}(J)$. Then $g_t\in \P_1^{\text{ent}}(J)$, $t\ge 0$.
\end{proposition}

\subsection{Stationary measures and transitions}\label{subsection:stationary}
Here, we would like to explore the possible stationary states of the nonlocal macroscopic equation \eqref{eq:macro} and provide an expression similar to the classical SIR model in order to determine the occurrence of an epidemic, or otherwise. 

As noted in Remark~\ref{rem:sus}, any agent that begins with the state $u=-1$ remains there for all times. Therefore, we expect $\delta_{-1}$ to be a natural stationary measure for \eqref{eq:macro}. In fact, it is not difficult to see that, if $\H$ is feasible and $\supp(\psi)\cap\supp(\chi)=\emptyset$, then $\delta_{u_*}$ and $\delta_{1}$ are also stationary measures. Indeed, since every stationary measure should satisfy
\[
 \int_J \left(\H'(u) - \bar\rho\M[g_\infty]\right) \partial_u\varphi (u)\,dg_\infty=0,\qquad\text{for all\, $\varphi\in\C_b(J)$},
\]
we simply substitute $g_\infty=\delta_\sigma$, $\sigma\in\{-1,u_*,1\}$ into the equation and use the fact that $\H'(\sigma)=0$, $\supp(\psi)\cap\supp(\chi)=\emptyset$, to verify its stationarity. The following result classifies all stable stationary states whenever $\H$ attains strict local minima at $u=-1$ and $u=1$.

\begin{theorem}\label{thm:stable}
 Let $\H$ be feasible, where $\H$ has strict local minima at $u\in\{-1,1\}$ and $\psi$, $\chi\in\Lip_b(S)$ with $\supp(\psi)\cap\supp(\chi)=\emptyset$. Then every measure of the form
 \[
  g_\infty = (1-\alpha)\delta_{-1} + \alpha\delta_1,\qquad \alpha\in[0,1],
 \]
 are stable stationary states of the nonlocal macroscopic equation \eqref{eq:macro}.
\end{theorem}
\begin{proof}
 To show the stability of $g_\infty$, we consider a perturbed measure of $g_\infty$ as initial condition and show that the solution $g_t$ of \eqref{eq:macro} converges towards $g_\infty$ as $t\to\infty$. More precisely, we consider the initial condition of the form
 \[
  g_\infty^\e = \varphi_\e \star g_\infty,
 \]
 with $\varphi_\e=\e^{-1}\varphi(u/\e)$, where $\varphi$ is any smooth positive symmetric mollifier. In order to determine $\e>0$ appropriately, we first establish neighborhoods $B_{-1}\subset \S$ and $B_1\subset \R$ around $u=-1$ and $u=1$, respectively, where $\H$ is strictly convex. Such neighborhoods exists since $\H$ has strict local minima at $u\in\{-1,1\}$. Consequently, we choose $\e>0$ such that $\supp(g_\infty^\e)\subset B_{-1}\cup B_1$.
 
 Now consider the Lagrangian solution $g$ corresponding to \eqref{eq:macro}, or equivalently,
 \[
  \frac{d}{dt}U_t(u) = -\H'(U_t(u)) + \bar\rho\M[g_t](U_t(u))
 \]
 with $u\in \supp(g_\infty^\e)$. Since $\S$ and $\R$ are disjoint sets, we may consider first $u\in \supp(g_\infty^\e)\cap \S\subset B_{-1}$. In this case, we have that
 \[
  \frac{d}{dt}U_t(u)\Big|_{t=0} = -\H'(u) + \bar\rho\M[g_0](u) = -\H'(u) < 0,
 \]
 which says that $U_t(u)$ remains in $B_{-1}$ for sufficiently small $t>0$, due to continuity. Analogously, we can show that $U_t(u)\in B_1$ for any $u\in \supp(g_\infty^\e)\cap \R\subset B_{1}$ when $t>0$ is sufficiently small. Therefore, the support of $g_t$ is contained within $B_{-1}\cup B_{1}$ for $t>0$ sufficiently small. By iterating this argument along the flow $U_t$, we have that $\supp(g_t)\subset B_{-1}\cup B_1$ for all times $t\ge 0$. Equivalently, we have that $U_t(u)\in B_1\cup B_{-1}$ for all $t\ge 0$, for any $u\in \supp(g_\infty^\e)$. 
 
 The previous discussion implies that $\M[g_t](v)=0$ for any $v\in J$, and hence
 \[
 \frac{d}{dt}U_t(u) = -\H'(U_t(u)),\qquad u\in \supp(g_\infty^\e).
 \]
 Taking the time derivative of $\H$ along the flow $U_t(u)$ yields
 \[
  \frac{d}{dt}\H(U_t(u)) = \H'(U_t(u))\frac{d}{dt}U_t(u) = -|\H'(U_t(u))|^2 < 0\qquad\text{for all\, $t\ge 0$},
 \]
 which says that the flow $U_t(u)$ minimizes $\H$ with time. Since $\H$ is strictly convex in $B_{-1}\cup B_1$, we have that $\H'(v)\ne 0$ for any $v\in B_{-1}\cup B_1$, $v\notin \{-1,1\}$. Thus,
 \[
  U_t(u) \longrightarrow \begin{cases}
   -1, & \text{for\, $u\in \supp(g_\infty^\e)\cap \S$} \\
   \phantom{-}1, & \text{for\, $u\in \supp(g_\infty^\e)\cap \R$}
  \end{cases}\qquad\text{as\, $t\to \infty$}.
 \]
 Note that $\supp(g_\infty^\e)\cap \S$ and $\supp(g_\infty^\e)\cap \R$ are disjoint, and so the mass of $g_t$ within $B_{-1}$ is $(1-\alpha)$, and $B_1$ is $\alpha$, due to conservation of mass. Consequently $g_t\to g_\infty$ in distribution as $t\to \infty$.
\end{proof}

\begin{remark}
 One easily verifies that stable stationary states for the mean-field equation \eqref{eq:fp} may be identified with the measure $f_\infty(dx,du)=\rho_{\text{stat}}(dx)\otimes g_\infty(du)$.
\end{remark}

We now proceed to derive an equivalent expression for the {\em basic reproduction number} $\mathfrak{R}_0$ present in the classical SIR model, which determines if a disease leads to an epidemic or otherwise. For the classical SIR model, the basic reproduction number $\mathfrak{R}_0$ is given by the formula $\mathfrak{R}_0=\beta S_0/\gamma$, where $\beta>0$ is the transmission rate, $\gamma$ the recovery rate, and $S_0$ the initial susceptible population. 

To provide a correspondence between the nonlocal macroscopic model \eqref{eq:macro} and the SIR model, we make simplifying assumptions on $\H$, $\psi$ and $\chi$. More precisely, we assume that
\begin{align}\label{eq:simplify}
 \psi=\1_\S^\e,\qquad \chi = c_\chi\1_\I^\e,\qquad \H' = \lambda\1_\S^\e -\gamma\1_\I^\e,
\end{align}
where $c_\chi, \lambda,\gamma$ are positive constants, and $\1_A^\e$ are mollified versions of the indicator function over a Borel set $A\subset J$ with $\supp(\1_A^\e)\subset A$. We further assume that $\bar u=1$, i.e., $\R=\{1\}$.

\begin{definition}
 We define the {\em effective transition} from the class of susceptible agents $\S$ to the class of infectious agents $\I$ as
 \[
  \mathcal{E}_t = \int_\S u\,dg_t - \int_\I u\, dg_t,
 \]
 for all times $t\ge 0$. Roughly speaking, $\mathcal{E}_t$ gives an indication of the probability of agents that lie within a infinitesimal neighborhood of $u_*$, i.e., around the point of transition. 
 
 Similar to the classical case, the disease is said to be {\em epidemic} if 
 \[
  \frac{d}{dt}\mathcal{E}_t\,\Big|_{t=0}>0,
 \]
 which suggests the presence of agents transitioning from class $\S$ to class $\I$. 
\end{definition}

By taking the temporal derivative of $\mathcal{E}_t$, we obtain
 \begin{align*}
  \frac{d}{dt}\mathcal{E}_t 
  &= -\int_\S \Big(\H' - \bar \rho \,\M[g_t]\Big)dg_t + u \Big(\H' g_t - \bar \rho \,\M[g_t]g_t\Big)\Big|_{-1}^{u_*} \\
  &\hspace*{6em}+ \int_\I \Big(\H' - \bar \rho \,\M[g_t]\Big)dg_t - u \Big(\H' g_t - \bar \rho \,\M[g_t]g_t\Big)\Big|_{u_*}^{1}\\
  &= -\lambda S_t + \bar \rho\, c_\chi S_t I_t - \gamma I_t, 
 \end{align*}
 where we denote $S_t = \int \1_\S^\e\,dg_t$ and $I_t = \int \1_\I^\e\, dg_t$. Consequently, we have
 \begin{align*}
  \frac{d}{dt}\mathcal{E}_t\,\Big|_{t=0} = \gamma I_0 (\mathfrak{R}_0-1),
 \end{align*}
 with the basic reproduction number $\mathfrak{R}_0= \bar\rho\,c_\chi S_0/\gamma - \lambda S_0/(\gamma I_0)$, which indicates that an epidemic only occurs when $\mathfrak{R}_0>1$. Notice that if $\lambda=0$, we recover the classical basic reproduction number.

 Summarizing the discussion above yields the following statement.
 
\begin{theorem}\label{thm:epidemic}
 Let $\H$, $\psi$ and $\chi$ be feasible and satisfy additionally \eqref{eq:simplify}. Then
 \[
  \frac{d}{dt}\mathcal{E}_t\,\Big|_{t=0} = \gamma I_0 (\mathfrak{R}_0-1),
 \]
 with the basic reproduction number $\mathfrak{R}_0= \bar\rho\,c_\chi S_0/\gamma - \lambda S_0/(\gamma I_0)$. 
 
 In particular, an epidemic occurs when $\mathfrak{R}_0>1$.
\end{theorem}

\section{Numerical Investigations}\label{sec:numerics}
We recall the two epidemiological models that will be under investigation within this section, namely the microscopic model
\[
 dX_t^i = \sqrt{2\sigma}\,dW_t^i,\qquad dU_t^i =-\H'(U_t^i)\,dt + F_N(X_t^i,U_t^i,{\bf X}_t,{\bf U}_t)\,dt,
\]
and the macroscopic model
\[
 \partial_t g_t -\partial_u \Big(\H' g_t - \bar \rho \,\M[g_t]g_t\Big) = 0.
\]
In all our numerical simulations, we consider the spatial dynamics to be within the bounded domain $\Omega=[0,1]^2$, with reflecting boundary conditions for the Wiener processes $W_t^i$. Furthermore, we only consider interactions of product form, i.e.,
\[
 \K(x,u,y,\nu)=\Phi(x-y)\psi(u)\chi(\nu),\qquad (x,u), (y,\nu)\in S.
\]

The standard {\em Euler--Maruyama} scheme was employed to solve the microscopic equations numerically \cite{kloeden2011numerical}. Appropriate step sizes were chosen to ensure stability of the explicit scheme. In the absence of noise, i.e., $\sigma=0$, we simply use the standard {\em explicit Euler} scheme. Since stochastic processes admit different solution paths for different realizations, we consider multiple realizations (often $M=100$ realizations) to obtain statistical information such as the mean and variance.

To compute the mean $m_t^\varphi$ of a given observable bounded $\varphi$, we use the well-known estimator
\[
 m^\varphi_t = \frac{1}{M}\sum\nolimits_{j=1}^M \varphi({\bf Z}_t^{(j)}),
\]
which ensures convergence towards the mean via the law of large numbers. Here, the superscript index $j$ represents the $j$th realization of the microscopic simulation. Typical observables we often use are the number of agents within the health classes $\S$, $\I$ and $\R$:
\[
 S_t = \sum\nolimits_{i=1}^N \1_\S(U_t^i),\qquad  I_t = \sum\nolimits_{i=1}^N \1_\I(U_t^i),\qquad  R_t = \sum\nolimits_{i=1}^N \1_\R(U_t^i).
\]
As an estimator for the variance, we choose the unbiased sample variance
\[
 \text{var}_t^\varphi = \frac{1}{M-1}\sum\nolimits_{j=1}^M \Big( \varphi({\bf Z}_t^{(j)}) - m^\varphi_t \Big)^2,\qquad s_t^\varphi = \sqrt{\text{var}_t^\varphi}.
\]
In all the plots below, we use the color {\em blue} to identify $S_t$, {\em red} for $I_t$ and {\em green} for $R_t$. The standard deviation for each observable $s_t^\varphi$ will be shown as shaded regions around its sample mean $m_t^\varphi$.

As for the numerical realization of the macroscopic equation, we employ a {\em Riemann} solver, or more precisely the {\em Harten--Lax--Leer} (HLL) Riemann solver \cite{toro2009}. In this method, an approximation for the intercell numerical flux is obtained directly, without the need to solve the local Riemann problems exactly. Therefore this is only an approximate {\em Godunov method}. The grid and time step sizes are chosen appropriately to satisfy the CFL condition demanded by the method.

All numerical simulations were implemented in {\sf python 2.7.6} with additional scientific computing packages, such as {\sf numpy} and {\sf scipy}.


\subsection{Comparison with the classical SIR model}\label{subsection:SIR}
Here, we address the question of whether the microscopic model (\ref{eq:micro}) can recover results obtained from the classical SIR model (cf.~Section~\ref{sec:intro}), at least in the qualitative sense. Obviously, this would require us to construct an appropriate potential landscape $\H$ and interaction force $\K$. However, while the microscopic model has multiple functions as parameters, the standard SIR model only has two parameters, namely $\beta$ and $\gamma$. For this reason, it is crucial to correctly understand and interpret these parameters accordingly. We further restrict the microscopic model to agents having no mobility $(\sigma=0)$ that are located on an equidistant grid in the domain $\Omega$. This reduces the model to a deterministic ordinary differential equation, apart from the initial distribution. 

As mentioned before, the parameter $\beta$ in the classical SIR model is known as the transmission rate, which depends on the probability of transmission $p$ and the average number of contacts per agent $C_0=C_0(N)$, irregardless of an agent's activity. More specifically, $\beta=p C_0$. Since the agents are stationary, it is possible to explicitly determine the number of contacts per agent. 

\begin{figure}[h]
 \includegraphics[width=0.25\textwidth]{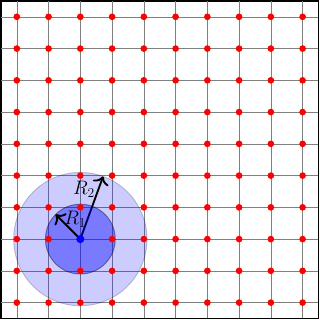}
 \caption{Agents on an equidistant grid in $\Omega$ with possible interaction regions.}\label{fig:grid}
\end{figure}

Using the indicator function $\1_{B_R(x)}$, we determine $C_0$ the number of agents within the vicinity of $x\in\Omega$ that are maximal radius $R$ away from $x\in\Omega$ (cf.~Fig.~\ref{fig:grid}). On the other hand, if we assume a uniform distribution for the spatial density, i.e., $\rho_{\text{stat}}\equiv 1/|\Omega|$, then $C_0$ may be considered as the product of the number density $N\rho_{\text{stat}}$ and the area of interaction indicated by $\1_{B_R(x)}$, i.e.,
\[
 C_0 = C_0(N) = N\int_\Omega \1_{B_R(x)}(y)\,d\rho_{\text{stat}}(y) = \frac{N}{|\Omega|}|B_R(x)| = N\pi R^2/|\Omega|,
\]
which depends explicitly on the number of agents. Considering the equidistant grid in $\Omega=[0,1]^2$ for any $N\gg 1$, we rescale the radius as $R=R_0/\sqrt{N-1}$, where $R_0>0$ is a fixed constant, in order to keep the number of individual contacts bounded as $N\to\infty$. More precisely, we have
\begin{align}\label{eq:contact}
 C_0(N) = \frac{N}{N-1}\frac{\pi}{|\Omega|} R_0^2\;\longrightarrow\; c_0:=\frac{\pi}{|\Omega|} R_0^2,\qquad\text{as\, $N\to\infty$}.
\end{align}
There are also other ways of scaling the contact rate (see, for example \cite{hu2013scaling}). However, this consideration is, in fact, the weak coupling scaling mentioned in Remark~\ref{rem:weak_coupling}, which allows for the mean-field limit. Indeed, if we set $\rho^N=\frac{1}{N}\sum\nolimits_j \delta_{X^j}$ as the empirical measure of locations, then
\[
 \sum\nolimits_j \1_{B_R(x)}(X^j) = N\rho^N(B_R(x))  \approx \rho^N(B_{R_0}(x)) = \frac{1}{N}\sum\nolimits_j \1_{B_{R_0}(x)}(X^j), 
\]
for a large number of agents $N\gg 1$. Hence, it makes sense to use the activation function
\[
 \Phi(x,y)=\1_{B_{R_0}(x)}(y),
\]
where $R_0$ is chosen appropriately, depending on $\beta$.

\begin{figure}[h]
 \includegraphics[width=0.4\textwidth]{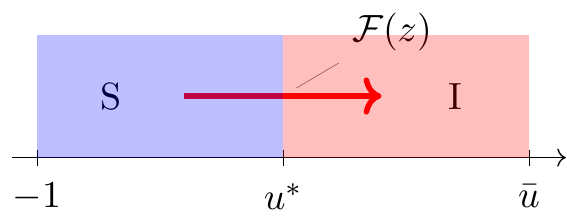}\vspace*{-1em}
 \caption{Interaction force acting on the class $\S$.}\label{fig:num-interaction}
\end{figure}

We now work towards identifying the transmission probability $p$ by considering the mean-field equation \eqref{eq:nonlinear} with the product distribution $f_t = \rho_{\text{stat}}\otimes g_t$ and $\psi=\1_\S^\e$, $\chi=c_\chi \1_\I^\e$, as adopted in \eqref{eq:simplify}. The expected effective intensity of interaction between the class of susceptible and infectious agents (cf.~Fig.~\ref{fig:num-interaction}) may then be computed as
\begin{align*}
 \E\left[\F[f_t](\X_t,\U_t)\right] &= c_\chi\iint_{S\times S} \1_{B_{R_0}(x)}(y) \1_\S^\e(u)\1_\I^\e(\nu) df_t(y,\nu)df_t(x,u) = c_\chi c_0 S_t I_t,
\end{align*}
where $S_t = \int \1_\S^\e\,dg_t$ and $I_t = \int \1_\I^\e\, dg_t$ provides the probabilities in classes $\S$ and $\I$, respectively. A direct comparison with the classical SIR model reveals the correspondence $\beta=c_\chi c_0$. However, since $c_0$ denotes the average number of contacts per agent, we have the relation $p=c_\chi$. Therefore, $c_\chi$ may also be considered as the probability of transmission of a specific disease. Let us summarize the discussion so far. Given $\beta=p C_0$ from the classical SIR model, we choose $R_0$ satisfying \eqref{eq:contact}, thereby yielding the interaction term
\[
 F_N(X^i,U^i,{\bf X},{\bf U}) = \frac{p}{N}\sum\nolimits_{j\ne i} \1_{B_{R_0}(X^i)}(X^j) \1_\S^\e(U^i)\1_\I^\e(U^j).
\]

As for the potential landscape $\H$, we first note that the classical SIR model does not describe the resistance of an agent towards an infection. Therefore, $\H'|_\S \equiv 0$. On the other hand, if an agent has been infected, i.e., $U^i\in \I\cup\R$ the interaction term vanishes. Hence, $\H|_{\I\cup\R}$ should describe the process of recovery. We further assume that $\H|_{\I\cup\R}$ is linear, with a maximum at $u=u_*$ and minimum at $u=1$ (cf.~Fig.~\ref{fig:potential}(b)). Then the evolution of an infected agent is given by
\[
 \frac{d}{dt}U_t^i = -\H'(U_t^i) = \lambda,\qquad U_0^i=u_*,
\]
with $\lambda>0$, where $-\lambda$ is the slope of $\H|_{\I\cup\R}$. Solving this equation gives $U_t^i = u_* + \lambda t$. Recalling the definition of the recovery rate $\gamma=1/\tau$, where $\tau>0$ denotes the mean waiting time until an infected individual recovers, we deduce $U_{\tau}^i=1$. Hence, we obtain the relation
\[
 \lambda = (1-u_*)/\tau = \gamma(1-u_*).
\]
Putting all conditions together, we end up with a piecewise linear potential landscape satisfying
\begin{align}\label{eq:landscape}
 \H' = \gamma(1-u_*)\1_{\I\cup \R}.
\end{align}
The piecewise linear potential landscape $\H_{\text{lin}}$ depicted in Fig.~\ref{fig:potential}, for instance, verifies the above requirements and provides a prototype for this comparison throughout this section. 

\begin{figure}[h]\centering
 \subfigure[Non-epidemic]{\includegraphics[width=0.4\textwidth]{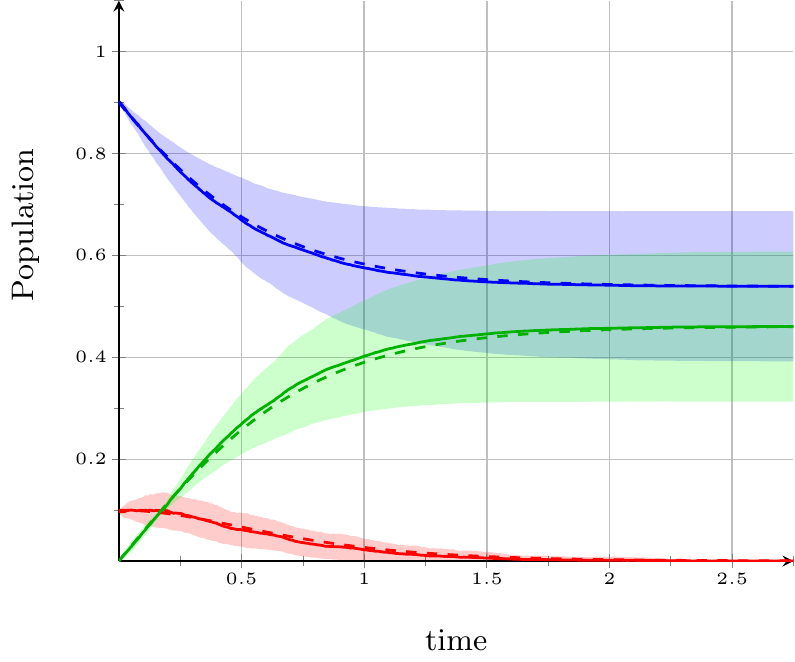}}\hspace*{0.1\textwidth}
 \subfigure[Epidemic]{\includegraphics[width=0.4\textwidth]{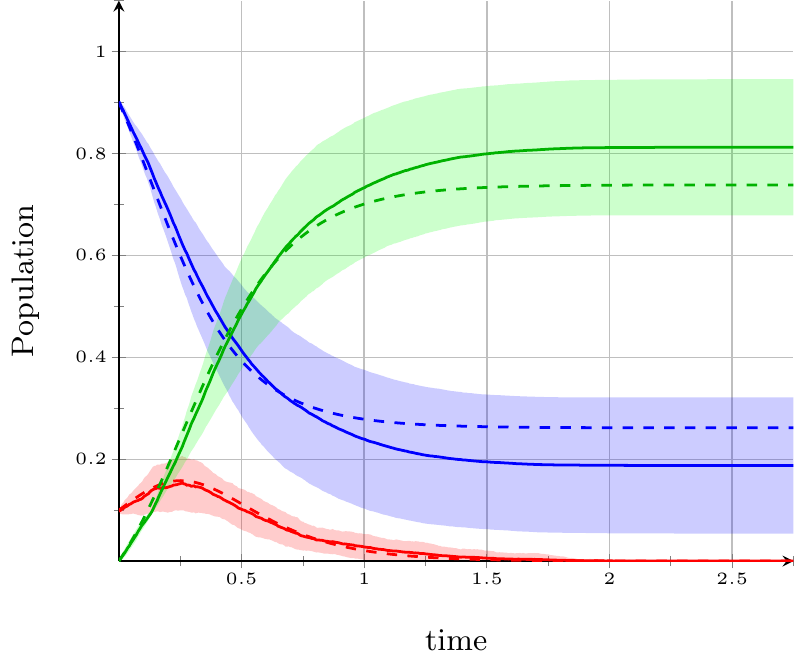}}
 \caption{Compartmental evolution for the microscopic model on an equidistant grid with $N=100$ agents and $M=100$ realizations in comparison with classical SIR model. The dashed line represents the classical SIR model.}\label{fig:sir}
\end{figure}

Fig.~\ref{fig:sir} shows an acceptable amount of similarity between the two models since the classical SIR model provides solutions that lie within the shaded region of the microscopic model. For the simulation in Fig.~\ref{fig:sir}, we distribute the agents' locations  on an equidistant grid, with 90\% of the agents having activity uniformly distributed in $\S$ and 10\% of the agents having activity uniformly distributed in $\I$. Other parameters used are $C_0=8$, $p=0.5$, $u_*=0$, $\bar u=0.3$ and $\gamma=1.5$.

\begin{remark}
 If we determine $C_0$ via a standard Gaussian distribution with variance $\sigma^2>0$, instead of using $\1_{B_R(x)}$, we obtain the formula
 \[
  C_0(N) = N\int_\Omega \exp\left(-\frac{|x-y|^2}{2\sigma^2}\right)d\rho_{\text{stat}}(y) = 2N\pi \sigma^2/|\Omega|,
 \]
 which depends again on the number of agents. To ensure that $C_0$ remains constant for any $N\in\N$, we rescale the variance as $\sigma^2=\sigma_0^2/N$, which yields $C_0= 2\pi\sigma_0^2/|\Omega|$. As opposed to the indicator function, the Gaussian distribution considers every agent as a neighbor. Neighbors that are closer are given more weight than those that are further away. This might be more appropriate whenever considering a domain $\Omega$ that represents, for example, an enclosed medium sized room. 
 
 \begin{figure}[h]\centering
  \subfigure[Non-epidemic]{\includegraphics[width=0.42\textwidth]{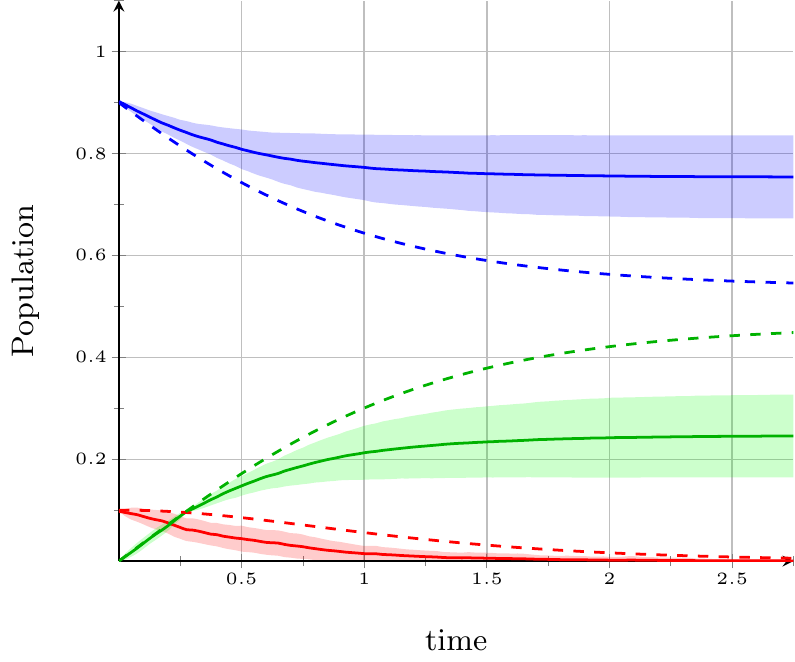}}\hspace*{2.5em}
  \subfigure[Epidemic]{\includegraphics[width=0.42\textwidth]{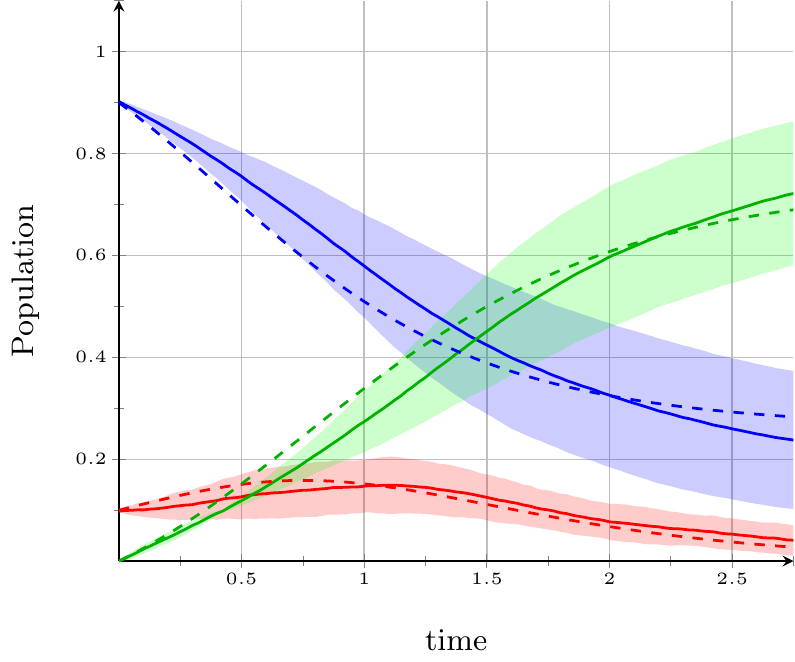}}
  \caption{Compartmental evolution for the microscopic model on an equidistant grid with $N=225$ agents, $M=100$ realizations and a Gaussian-type activation function in comparison with classical SIR model.}\label{fig:sir-gauss}
 \end{figure}
 
 Fig.\ref{fig:sir-gauss} depicts the comparison between the microscopic model with a Gaussian-type activation function $\Phi$, as described above. One observes a slight disparity between the two models, especially in the non-epidemic case, where the number of recovered agents are fewer than the susceptible ones, in contrast to the classical SIR model. Nevertheless, the qualitative behavior of the solutions do coincide to some extend.
\end{remark}

\subsection{Links between the microscopic and macroscopic models}\label{subsection:links}

Another interesting context for the agent-based model is its connection to its macroscopic counterpart. We now investigate this relationship via two ways, namely the direct link, and the mobility link. Unless stated otherwise, we consider a linear potential landscape $\H$ satisfying \eqref{eq:landscape}, and $\psi = \1_\S^\e$, $\chi=c_\chi \1_\I^\e$ for the simulations within this subsection. Furthermore, the agents' locations are initially distributed on an equidistant grid, with 80\% of the agents having activity uniformly distributed in $\S$ and 20\% of the agents having activity uniformly distributed in $\I$. 

\subsubsection{Direct link}
Looking back at the derivation of the macroscopic model, we first derived the mean-field equation by passing to the limit $N\to\infty$, and thereafter the spatial activation function $\Phi$ was removed in the process. Therefore, we will need to look for an appropriate $\Phi$ for the microscopic model for comparison. In fact, the spatial scaling $\tilde x\sim \e x$ conducted in Model~\ref{model:macro} contracts the bounded domain $\Omega$ into a spatially concentrated point as $\e \to 0$. Hence, the macroscopic model may also be seen as a {\em complete mixture model}, where the support of $\Phi$ is the entire domain, i.e., $\supp(\Phi)=\Omega$. Consequently, choosing $\Phi \equiv 1$ results in an agent based model, which is independent of spatial resolution. Since the spatial configuration is obsolete in this case, we may consider any spatial location for the agents.

\begin{figure}[h]\centering
 \subfigure[$N=100$ agents]{\includegraphics[width=0.42\textwidth]{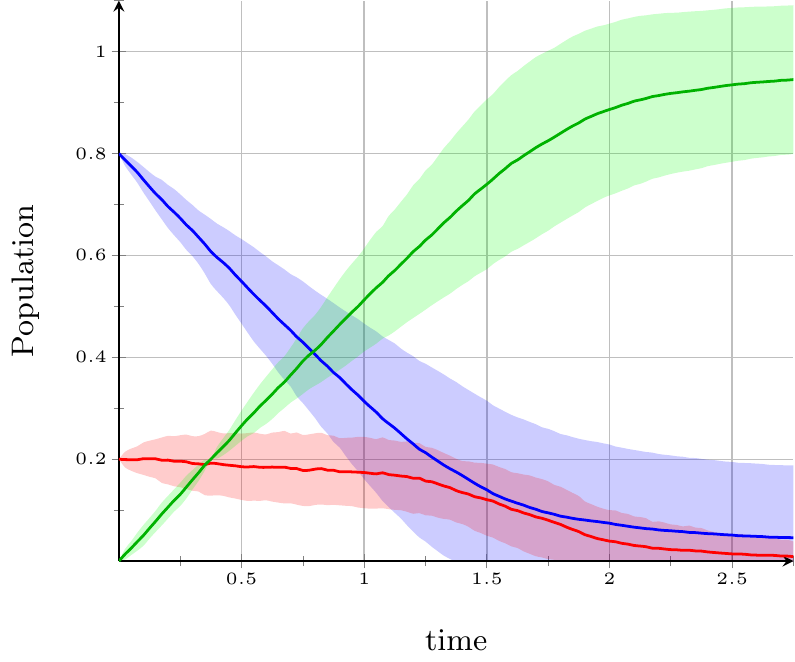}}\hspace*{2.5em}
 \subfigure[$N=400$ agents]{\includegraphics[width=0.42\textwidth]{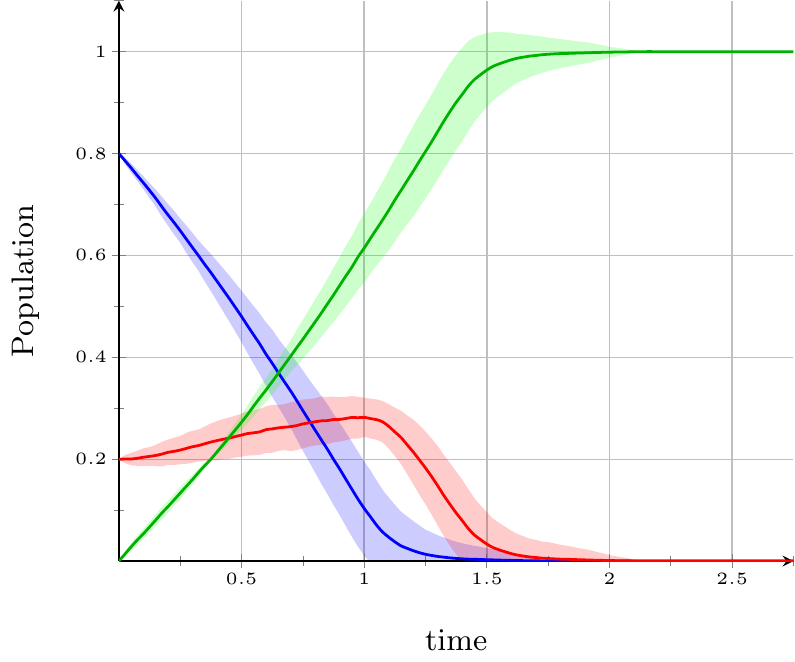}}
 \subfigure[$N=1600$ agents]{\includegraphics[width=0.42\textwidth]{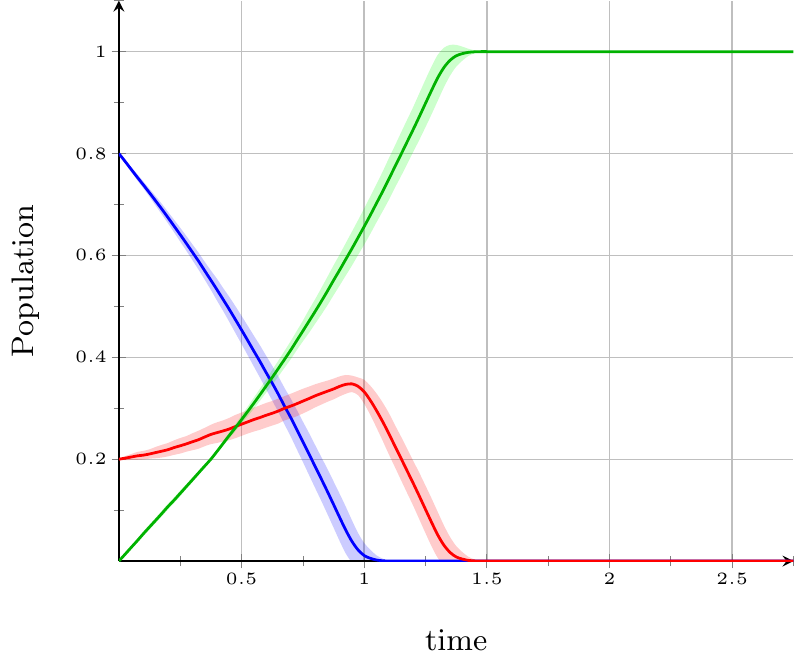}}\hspace*{2.5em}
 \subfigure[Macroscopic model]{\includegraphics[width=0.42\textwidth]{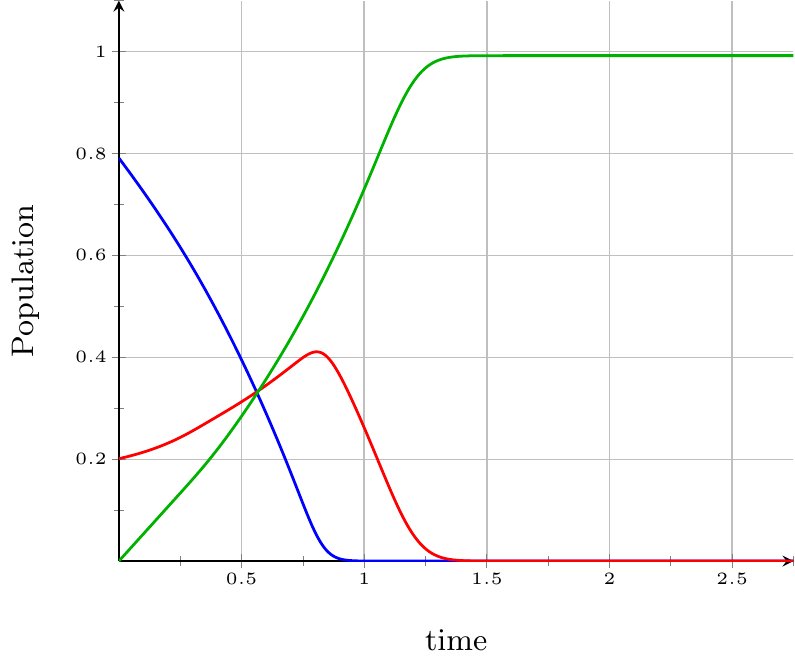}}
 \caption{Compartmental evolution for the microscopic model on an equidistant grid with increasing number of agents in comparison with the evolution generated by the macroscopic model. Other parameters used are $C_0=8$, $c_\chi=0.5$, $u_*=0$, $\bar u=0.3$ and $\gamma=1.5$.}\label{fig:directlink}
\end{figure}

As seen in Fig.~\ref{fig:directlink}, the solution provided by the microscopic model evidently converges to the solution of the macroscopic model as $N\to\infty$. This verifies on one hand the mean-field limit discussed in Section~\ref{sec:mean-field}, as well as the choice $\Phi\equiv 1$. To supplement the validation, we investigate the behavior of the probability distribution corresponding to the microscopic model with $N=1600$ agents and the macroscopic model, respectively. From this point of view, we recover the complete information concerning the temporal evolution of activity, which provides comprehensive behavior of transitions between the three health classes $\S$, $\I$ and $\R$.

\begin{figure}[h]\centering
 \subfigure[$t=0$]{\includegraphics[width=0.24\textwidth]{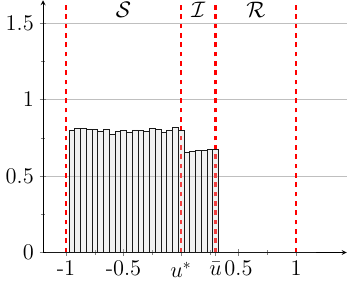}}
 \subfigure[$t\approx 0.6$]{\includegraphics[width=0.24\textwidth]{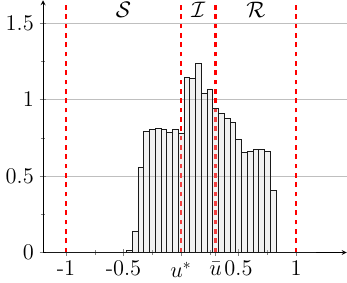}}
 \subfigure[$t\approx 0.75$]{\includegraphics[width=0.24\textwidth]{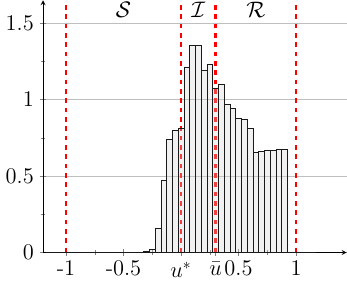}}
 \subfigure[$t\approx 1.125$]{\includegraphics[width=0.24\textwidth]{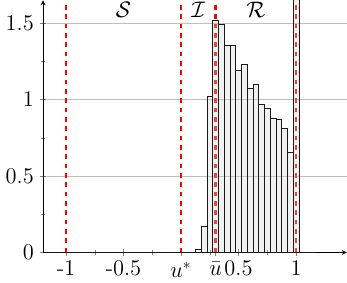}}
 \caption{Marginal distribution in the activity space of the microscopic model with $N=1600$ agents with parameters as used in Fig.~\ref{fig:directlink}.}\label{fig:microdensity}
\end{figure}

\begin{figure}[h]\centering
 \subfigure[$t=0$]{\includegraphics[width=0.24\textwidth]{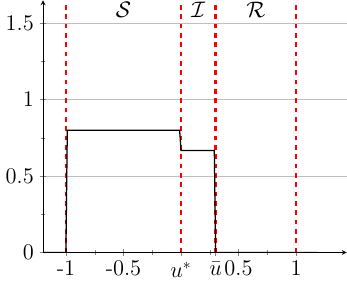}}
 \subfigure[$t\approx 0.6$]{\includegraphics[width=0.24\textwidth]{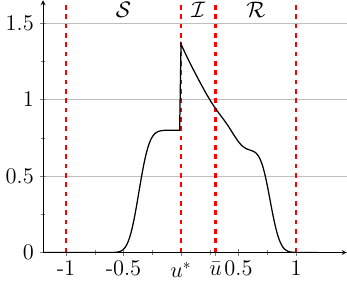}}
 \subfigure[$t\approx 0.75$]{\includegraphics[width=0.24\textwidth]{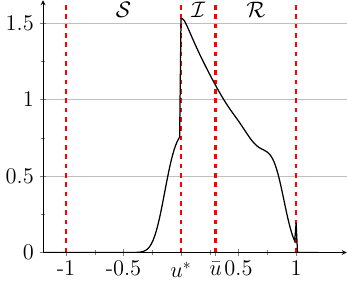}}
 \subfigure[$t\approx 1.125$]{\includegraphics[width=0.24\textwidth]{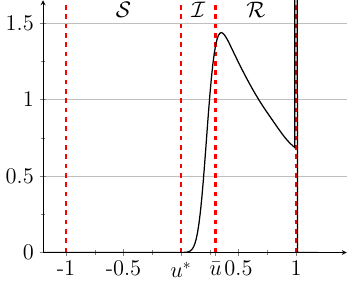}}
 \caption{Distribution $g_t$ corresponding to the macroscopic model \eqref{eq:macro} with parameters as used in Fig.~\ref{fig:directlink}.}\label{fig:macrodensity}
\end{figure}

A quick comparison of the figures shown in Fig.~\ref{fig:microdensity} and Fig.~\ref{fig:macrodensity} clearly shows evidence of conformity between the two models, even for the complete distribution sense. Fig.~\ref{fig:macrodensity} additionally depicts the convergence towards the stationary state $g_\infty=\delta_{1}$ derived in Section~\ref{subsection:stationary}.

\subsubsection{Mobility link}
In this case, we consider spatial movements of the agents in the form of the standard Wiener process within the bounded domain $\Omega=[0,1]^2$, i.e., we consider the complete microscopic equation \eqref{eq:micro} with reflective boundary conditions for the Wiener process. The simulations in this part concentrates on the connection between the stationary microscopic model introduced in Section~\ref{subsection:SIR}, which coincides with the classical SIR model, and the complete mixture model described in the previous case. Practically speaking, the presence of mobility 'interpolates' between theses two scenarios, with $\sigma\in[0,1]$, i.e., the intensity of mobility being the interpolation parameter, as may be seen in Fig.~\ref{fig:mobilitylink}.

\begin{figure}[h]\centering
 \subfigure[$\sigma=0$]{\includegraphics[width=0.42\textwidth]{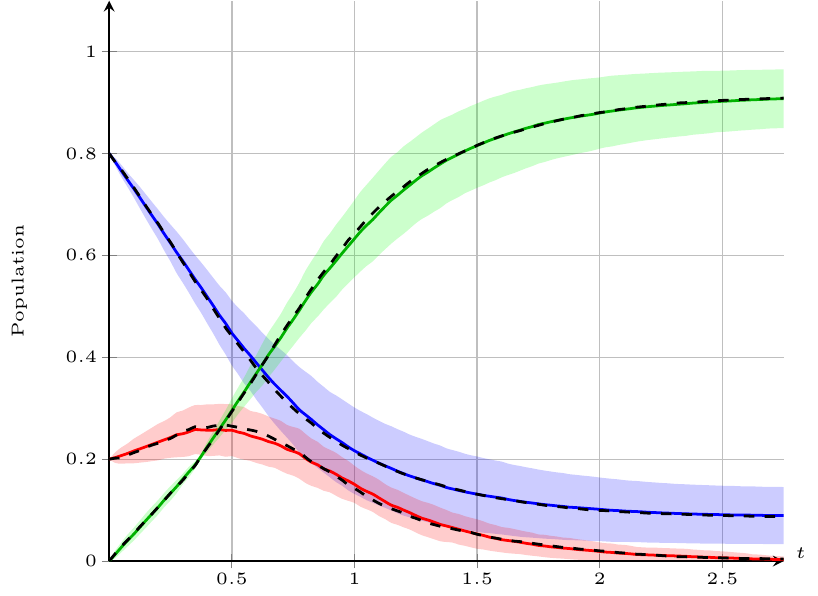}}\hspace*{2.5em}
 \subfigure[$\sigma=0.01$]{\includegraphics[width=0.42\textwidth]{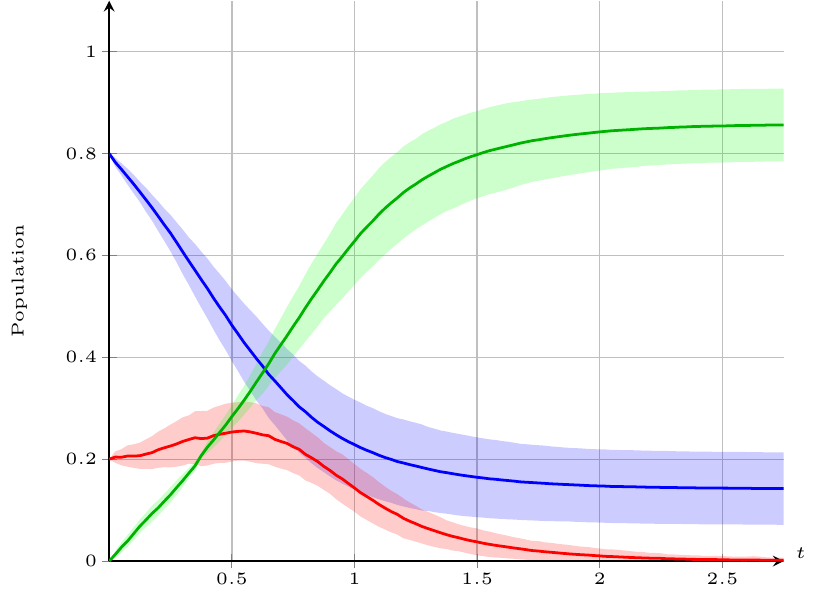}}
 \subfigure[$\sigma=0.05$]{\includegraphics[width=0.42\textwidth]{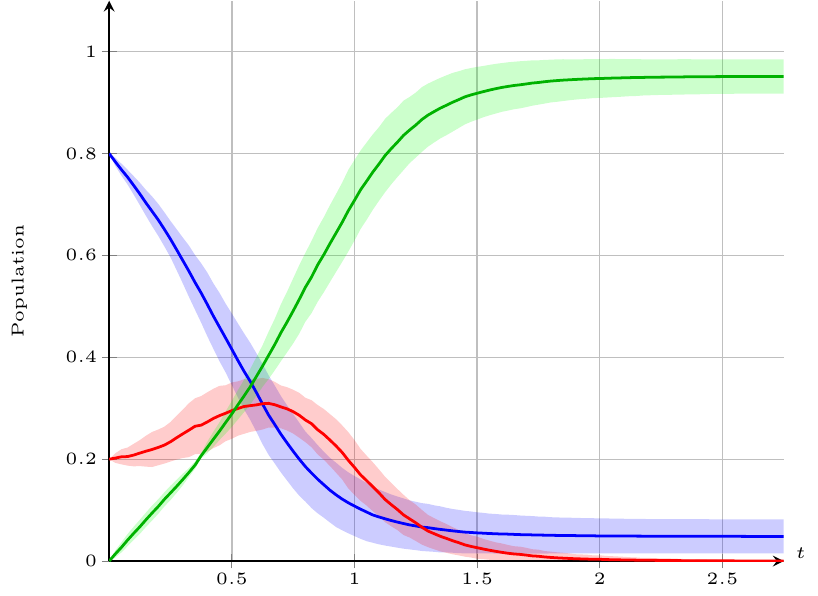}}\hspace*{2.5em}
 \subfigure[$\sigma=1$]{\includegraphics[width=0.42\textwidth]{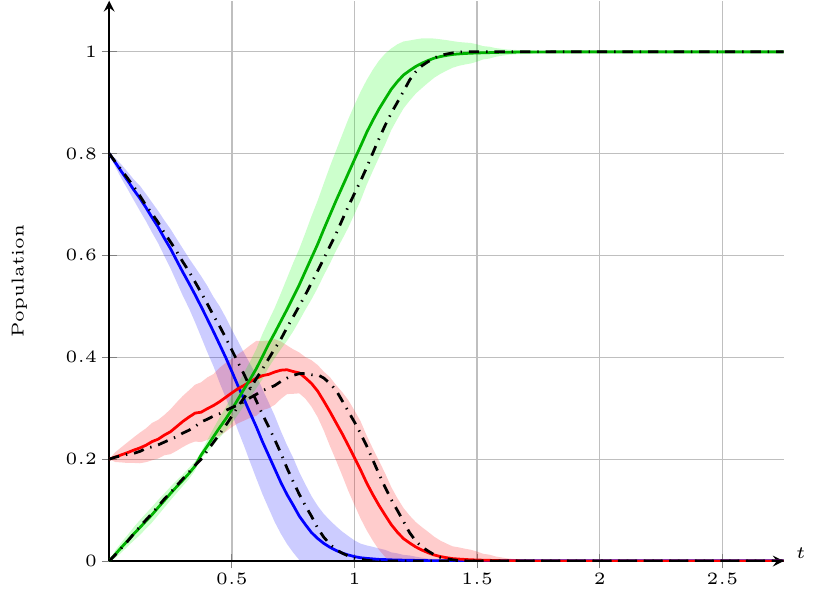}}
 \caption{Convergence of the microscopic model with $N=225$ for increasing mobility $\sigma$ towards the complete mixture model. The dashed line represents the stationary microscopic model, while the dotted line depicts the complete mixture model. Other parameters used are $C_0=8$, $c_\chi=0.5$, $u_*=0$, $\bar u=0.3$ and $\gamma=1.5$.}\label{fig:mobilitylink}
\end{figure}

The results shown in Fig.~\ref{fig:mobilitylink} suggest that increasing mobility increases the mixture of susceptible and infectious agents. This observation is to be expected since the increase in mobility speeds up the rate at which the spatial distribution reaches uniformity, which expresses that the amount of time an agent remains in a particular location is the same amount of time it remains everywhere else. This means that the Wiener process with $\sigma=1$ ensures that every susceptible agent has the same probability to meet an infectious agent as with other susceptible ones, and vice versa. Hence, a microscopic system with a high mobility intensity may just as well be modeled by the macroscopic model which demands less computational effort, since the spatial activation function $\Phi$ in the microscopic model has to be recomputed at every time step.

\subsection{Microscopic model: A spatially inhomogeneous setting}\label{subsection:inhomogeneous}

We finally consider a case the classical SIR model is unable to capture, namely the case where initial distribution of susceptible and infectious agents are no longer uniformly distributed within $\Omega$. To simplify the presentation, we consider an inhomogeneous setting  (in initial activity/health status) for the microscopic model with locations distributed equidistantly in $\Omega$.

\begin{figure}[h]\centering
 \subfigure[$t=0$]{\includegraphics[width=0.42\textwidth]{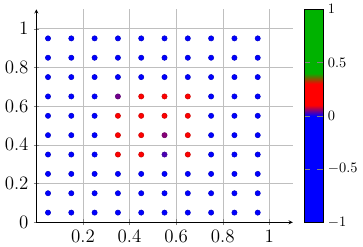}}\hspace*{2.5em}
 \subfigure[$t=0.5$]{\includegraphics[width=0.42\textwidth]{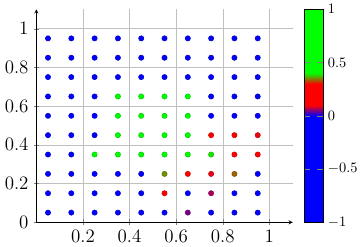}}
 \subfigure[$t=1$]{\includegraphics[width=0.42\textwidth]{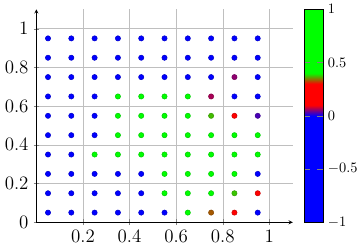}}\hspace*{2.5em}
 \subfigure[$t=3$]{\includegraphics[width=0.42\textwidth]{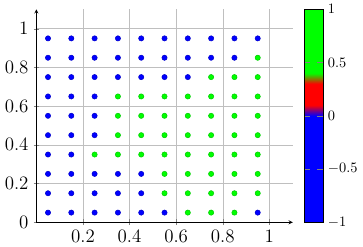}}
 \caption{Microscopic model with $N=100$ agents on an equidistant grid and a clustered initial distribution of activity. Other parameters used are $C_0=8$, $c_\chi=0.5$, $u_*=0$, $\bar u=0.3$ and $\gamma=1.5$.}\label{fig:inhomogeneous}
\end{figure}

In Fig.~\ref{fig:inhomogeneous}, one clearly observes that the distribution of location for infectious agents remains inhomogeneous at time $t>0$. Note that the activity of susceptible agents are evenly distributed in activity space. In this particular example, the susceptible agents bordering the upper region of the infectious agents at time $t=0$ have an activity close to being immune, i.e., $u\approx -1$, and therefore remain susceptible for all times $t\ge 0$.

\begin{figure}[h]
 \includegraphics[width=0.4\textwidth]{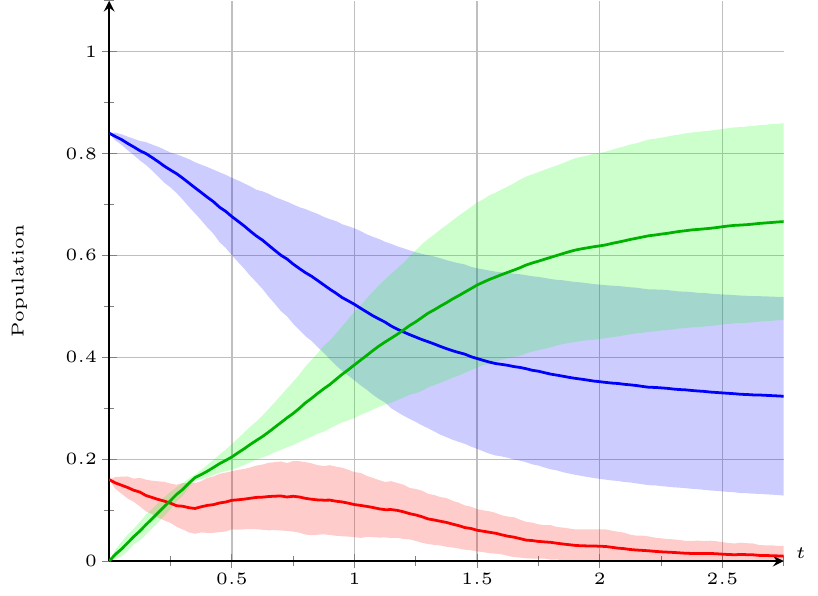}
 \caption{Compartmental evolution corresponding to Fig.~\ref{fig:inhomogeneous}.}\label{fig:sir-inhomogeneous}
\end{figure}

Fig.~\ref{fig:sir-inhomogeneous} depicts the compartmental evolution of the microscopic simulation shown in Fig.~\ref{fig:inhomogeneous}. Plots of this sort cannot be observed when using the classical SIR model, especially in the qualitative behavior of classes $\I$ and $\R$ between $t=0$ and $t=0.5$.

\subsection{Verification of Theorem~\ref{thm:epidemic}}\label{subsection:verify}
Here, we verify the validity of Theorem~\ref{thm:epidemic}. Therefore, following the assumptions made on the potential landscape and interaction term in Section~\ref{subsection:stationary}, we set
\[
 \psi=\1_\S^\e,\qquad \chi = c_\chi\1_\I^\e,\qquad \H' = \lambda\1_\S^\e -\gamma\1_\I^\e,
\]
with $\S=(-1,0)$ and $\I=(0,1)$, for different choices of parameters $c_\chi, \lambda,\gamma$ and initial conditions $S_0,I_0$, to obtain the cases $\mathfrak{R}_0=\bar\rho c_\chi S_0/\gamma - \lambda S_0/(\gamma I_0)>1$ (epidemic) or $\mathfrak{R}_0<1$ (non-epidemic). As in Section~\ref{subsection:links}, the agents' locations are initially distributed on an equidistant grid, with 80\% of the agents having activity uniformly distributed in $\S$ and 20\% of the agents having activity uniformly distributed in $\I$. In all cases, we set $\bar\rho=2$, to allow for larger values of $\lambda$ and $\gamma$, thereby speeding up the evolution, while keeping $c_\chi\in(0,1)$ fixed. We also set $\lambda=\gamma$, which leaves only $\gamma$ to be varied. In this case, the basic reproduction number simplifies to $\mathfrak{R}_0= \bar\rho c_\chi S_0/\gamma - S_0/I_0$.

 \begin{figure}[h]\centering
  \subfigure[Compartmental evolution]{\includegraphics[width=0.42\textwidth]{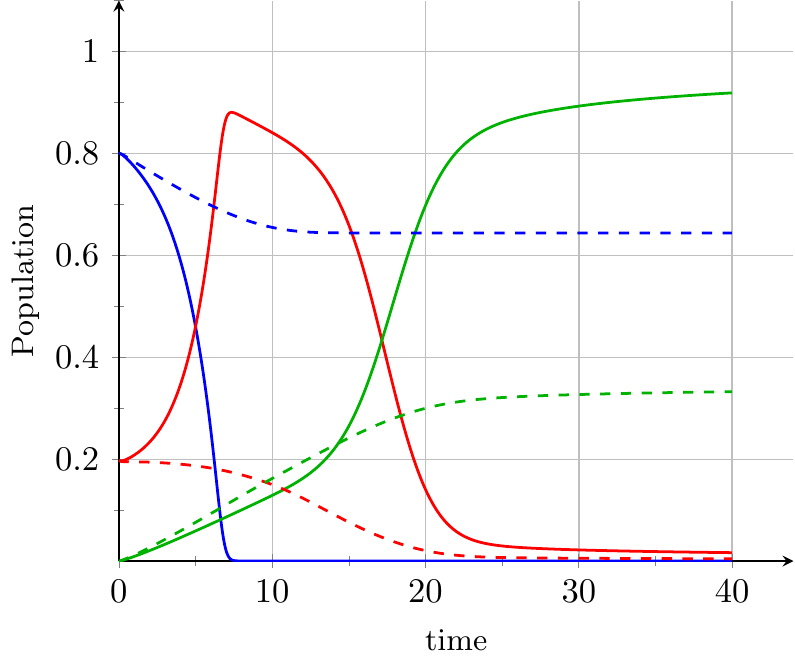}}\hspace*{2.5em}
  \subfigure[Effective transition $\mathcal{E}_t$]{\includegraphics[width=0.42\textwidth]{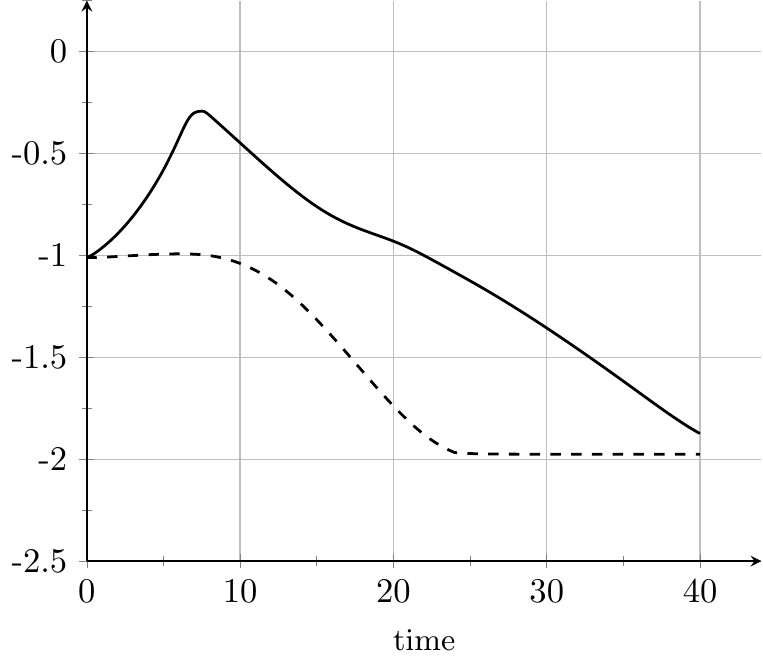}}
  \caption{Evolution of the macroscopic model with $c_\chi=0.3$. The solid lines depict the case $\mathfrak{R}_0=2$ (epidemic) with $\gamma=0.08$, while the dashed lines represent the case $\mathfrak{R}_0=0.8$ (non-epidemic) with $\gamma=0.1$.}\label{fig:verify}
 \end{figure}
 
 Fig.~\ref{fig:verify} provides the verification of Theorem~\ref{thm:epidemic}. In Fig.~\ref{fig:verify}(a), one clearly observes both cases, namely epidemic and non-epidemic, when the basic reproduction number $\mathfrak{R}_0$ is either greater than $1$ or less than $1$. This figure also affirms the stationary states suggested in Section~\ref{subsection:stationary}. Indeed, for $\mathfrak{R}_0>1$, we see that $g_t$ converges towards $g_\infty=\delta_1$ as $t\to\infty$, while for $\mathfrak{R}_0<1$, the stationary state is $g_\infty=(1-\alpha)\delta_{-1} + \alpha\delta_1$ with $\alpha\approx 0.3$. Fig.~\ref{fig:verify}(b), on the other hand, provides the evolution of the effective transition $\mathcal{E}_t$. As expected the temporal derivative of $\mathcal{E}_t$ at $t=0$ has the correct sign as indicated by Theorem~\ref{thm:epidemic}. Essentially, Fig.~\ref{fig:verify}(b) describes the complete behavior of the transition from the class $\S$ to class $\I$ within an infinitesimal neighborhood of the point $u_*\in J$. 

\section{Summary and outlook}\label{sec:conclusion}

In this paper, we successfully developed a model for mathematical epidemiology with spatial resolution, which also allows for a more detailed description of an agent's health status. This clearly paves a way for a more general description of disease transmission, thereby rendering it possible not only to determine the total number of individuals in a certain class of health state, but would also assist in locating the source of a disease. Due to the diversity of the parameters involved in its derivation, the model is able to describe various situations. However, this flexibility becomes also a drawback since the specification of these parameters is not easily accessible and therefore deserves further investigation. 

We further provided a recipe for deriving spatial macroscopic models for disease dynamics via passage to the mesoscopic scale. For our specific macroscopic model, we were able to derive a quantity that reflects upon the basic reproduction number $\mathfrak{R}_0$, which determines the possibility of an outburst. We also showed that stable stationary states exist for the macroscopic equation and that these stationary states are of the form $g_\infty = (1-\alpha)\delta_{-1}+\alpha\delta_1$, $\alpha\in[0,1]$. Unfortunately, the determination of $\alpha$, depending on the initial distribution $g_0$ remains an open problem and would therefore require further investigation. Nevertheless, one could, at this point, work on identifying the parameter $\alpha$ by means of inverse problems methods (cf.~\cite{tse2012identification}).

An obvious extension of the current microscopic model would be to incorporate spatial demographic information, as well as spatial interactions among agents. In fact, this was already pointed out in Remark~\ref{rem:ornstein}. More specifically, one may consider the interacting system
\[
 dX_t^i = -\Big[\nabla_x V(X_t^i)+ \frac{1}{N} \sum\nolimits_{j\ne i} U(X_t^i,X_t^j)\Big]\,dt + \sqrt{2\sigma}\,dW_t^i,
\]
where $V\colon\rr^d\to \rr$ describes the landscape of an area of a populated region, and $U\colon\rr^d\times\rr^d\to\rr$ is an interaction potential which may be both attracting and repulsive. The convex neighborhood of the local minima of $V$ provide areas that model higher concentration of population, which have a reduced communication between the clusters but still allow for transitions between these clustered populations. This may represent, for example, cities and meeting points. Such potentials may also be used to model special paths on which agents may travel, such as the migration of animals.

Another possible extension is to include different types of agents. Such models may be used to describe vector-based transmitted diseases such as malaria, dengue fever and the Zika virus. In this case, the importance of spatial inhomogeneity is indispensable. The mean-field and macroscopic equations corresponding to such systems are then coupled partial differential equations.

An interesting aspect for modification is the activity set $J$, which was fixed as $J=[-1,1]$ in this paper. Using a different set $J$ could lead to different dynamics for the activity variable. For instance, we may take $J=\mathbb{S}^1=\{x\in\rr^2\,|\, |x|=1\}$ as the unit circle in $\rr^2$. On this activity set, one can allow for recovered agents to become susceptible again after having been infected, thereby leading to generalization of the well-known classical SIS model. The spatially homogeneous nonlocal macroscopic equation analogous to \eqref{eq:macro} will then be posed on $J=\mathbb{S}^1$, or equivalently on $[-1,1]$ with a periodic boundary condition for $g_t$ on $[-1,1]$, i.e., $g_t(-1)=g_t(1)$ for all times $t\ge 0$.

All in all, the basic models introduced in this paper has paved a way to further generalizations that should be numerically investigated and thoroughly analyzed on every scale.

\appendix
\section{Proof of Theorem~\ref{thm:mean-field}}\label{append:proof}
 Without loss of generality, we may suppose $Z_0^i = \Z_0^i$, since they are identically distributed. Then, taking the difference of the solutions leads to
 \begin{align*}
  d(X_t^i - \X_t^i) = 0,\qquad d(U_t^i - \U_t^i) = -a_N^i\,dt + b_N^i\,dt,
 \end{align*}
 where the last two terms are given by
 \begin{align*}
 a_N^i &= \H'(U_t^i) - \H'(\U_t^i), \\
 b_N^i &= \frac{1}{N}\sum\nolimits_{j\ne i} \K(Z _t^i,Z _t^j) - \int_{S}\K(\Z_t^i,z')f_t(dz').
 \end{align*}
 The term $b_N^i$ may be further decomposed to obtain $b_N^i = c_N^i + d_N^i$, where
 \begin{align*}
  c_N^i &= \frac{1}{N}\sum\nolimits_{j\ne i} \Big[\K(Z_t^i,Z_t^j) - \K(\Z_t^i,\Z_t^j)\Big], \\
  d_N^i &= \frac{1}{N}\sum\nolimits_{j\ne i} \K(\Z_t^i,\Z_t^j) - \int_{S}\K(\Z_t^i,z')f(dz').
 \end{align*}
 Morever, the sums may be extended to sums over all of $j$ since $\K$ is feasible, i.e., $\K(x,u,y,u)=0$ for any $x,y\in\Omega$, $u\in J$. We now investigate the terms separately.
 
 We begin with the term $a_N^i$ that is easily estimated due to the regularity of $\H$,
 \[
  -\E[\langle U_t^i - \U_t^i, a_N^i\rangle] \le c_{\H}\E[|Z_s^i - \Z_s^i|^2].
 \]
 As for $c_N^i$, we simply use the Lipschitz continuity of $\K$ to obtain
 \[
  \E[\langle U_t^i - \U_t^i, c_N^i\rangle] \le c_{\K}\bigg[ \E[|Z_s^i - \Z_s^i|^2] + \frac{1}{N}\sum\nolimits_{j=1}^N \E[|Z_s^j - \Z_s^j|^2] \bigg].
 \]
 To estimate $d_N^i$, we first define
 \[
  \kappa_t^i(\Z_t^j) := \K(\Z_t^i,\Z_t^j)-\E[\K(\Z_t^i,\Z_t^j)] = \K(\Z_t^i,\Z_t^j) - \int_{S} \K(\Z_t^i,z')f_t(dz'),
 \]
 Clearly $\E[\kappa_t^i(\Z_t^j)|\Z_t^i]=0$ for any $j\ne i$. Furthermore, we have
 \[
  \E[\kappa_t^i(\Z_t^j)\kappa_t^i(\Z_t^k)] = \E[\E[\kappa_t^i(\Z_t^j)\kappa_t^i(\Z_t^k)|\Z_t^i]]  = \E[\E[\kappa_t^i(\Z_t^j)|\Z_t^i]\E[\kappa_t^i(\Z_t^k)|\Z_t^i]] =0,
 \]
 since the processes $\Z_t^j$ and $\Z_t^k$ are independent for $j\ne k\ne i$. Consequently
 \begin{align*}
  \E\left[\left|\frac{1}{N}\sum\nolimits_{j\ne i} \kappa_t^i(\Z_t^j)\right|^2\right] &= \frac{1}{N^2} \sum\nolimits_{j,k\ne i}\E[ \kappa_t^i(\Z_t^j)\kappa_t^i(\Z_t^k)] = \frac{N-1}{N^2} \E[\kappa_t^i(\Z_t^j)^2]\\
  &\le \frac{N-1}{N^2}\iint_{S\times S} |\K(z,z')|^2f_t(dz')f_t(dz).
 \end{align*}
 Therefore, we obtain, by Young's inequality, the estimate
 \begin{align*}
  \E[\langle Z_t^i - \Z_t^i, d_N^i\rangle] &\le \frac{1}{2}\E[|Z_t^i - \Z_t^i|^2] + \frac{1}{2}\E[|d_N^i|^2] \\
  &\le \frac{1}{2}\E[|Z_t^i - \Z_t^i|^2] + \frac{N-1}{2N^2}\iint_{S\times S} |\K(z,z')|^2f_t(dz')f_t(dz) \\
  &\le \frac{1}{2}\E[|Z_t^i - \Z_t^i|^2] + \frac{c_0}{N}.
 \end{align*}
 Now, set $Y_t^i = \E[|Z_t^i-\Z_t^i|^2]$. Then, by It\^o's calculus and the estimates above, we obtain
 \[
   \frac{d}{dt}Y_t^i = 2\,\E\left[ \langle Z_t^i-\Z_t^i, a_N^i + c_N^i + d_N^i\rangle\right] \le c \left[ Y_t^i + \frac{1}{N}\sum\nolimits_{j=1}^N Y_t^j + \frac{1}{N}\right].
 \]
Averaging over $1\le i\le N$ gives
 \[
  \frac{d}{dt}Y_t^{(N)}:=\frac{d}{dt}\frac{1}{N}\sum\nolimits_{i=1}^N Y_t^i \le \tilde c \left[ Y_t^{(N)} + \frac{1}{N}\right].
 \]
An application of the Gronwall inequality yields
 \[
  Y_t^N \le \frac{\tilde c}{N}t e^{\tilde c t}.
 \]
 Substituting this into the inequality for $Y_t^i$ and using Gronwall's inequality again yields
 \[
  \sup\nolimits_{t\in[0,T]}Y_t^i \le \frac{\tilde c}{N} T e^{2\tilde cT},
 \]
 which is precisely the required estimate for any $1\le i\le N\in\N$.\hfill\endproof

\bibliographystyle{plain}
\bibliography{disease}

\end{document}